\documentclass[11pt]{article}
\usepackage{amsmath,amsthm,amsfonts,amscd,bezier}
\usepackage[usenames]{color}
\usepackage[normalem]{ulem}
\usepackage{latexsym,amssymb}
\usepackage{lineno}

\newtheorem{theorem}{Theorem}[section]
\newtheorem{lemma}[theorem]{Lemma}
\newtheorem{example}[theorem]{Example}
\newtheorem{definition}[theorem]{Definition}

\newtheorem{remark}[theorem]{Remark}
\newtheorem{proposition}[theorem]{Proposition}

\begin{document}

\pagestyle{plain} \pagenumbering{arabic}

\title{Zariski invariant for quasi-ordinary hypersurfaces}
\author{Barbosa, R. A. and Hernandes, M. E.
\thanks{The first author
was partially supported by CAPES and the second one by CNPq.}\ \
\thanks{Barbosa, R. A. email: barbosa.rafael@ufms.br; Hernandes, M. E. email: mehernandes@uem.br }}
\date{ \ }
\maketitle

\begin{center}
	

key words: Quasi-ordinary hypersurface, Analytical invariants, Generalized Zariski exponents.
\end{center}

\begin{abstract}
	We introduced an $\tilde{\mathcal{A}}$-invariant for quasi-ordinary parameterizations and we consider it to describe quasi-ordinary surfaces with one generalized characteristic exponent admitting a countable moduli.
\end{abstract}

\section{Introduction}

In \cite{zariski-inv} Zariski introduces an analytic invariant for irreducible plane curves (plane branches) that can be determined directly from a parameterization such an invariant is known as {\it Zariski invariant} or {\it Zariski exponent}. Considering the topological class and the Zariski exponent it is possible to describe all plane branches admitting a zero-dimensional moduli space (see \cite{gaffney} and \cite{handbook}, for instance). 

Analytic plane curves are particular cases of quasi-ordinary hypersurfaces. An analytic germ $\left( \mathcal{X}, 0 \right) \subset \left( \mathbb{C}^{r+1}, 0 \right )$ of hypersurface is called a {\it quasi-ordinary}, shortly q.o.h., if there are local coordinates $(\underline{X},X_{r+1}):=\left( X_{1}, \ldots , X_{r}, X_{r+1} \right) $ such that $\left( \mathcal{X}, 0 \right)$, in these coordinates, is given by $\left\lbrace (\alpha_1,\ldots ,\alpha_{r+1}) \in \mathbb{C}^{r+1}; f(\alpha_1,\ldots ,\alpha_{r+1}) = 0 \right\rbrace $ where $f \in \mathbb{C}\{\underline{X} \} \left[ X_{r+1} \right]$ is a Weierstrass polynomial with discriminant $\Delta_{X_{r+1}}f=\underline{X}^{\delta}\cdot u:=X_1^{\delta_1}\cdot\ldots\cdot X_{r}^{\delta_r}\cdot u$
for some unit $u\in\mathbb{C}\{\underline{X}\}$ and $\delta=(\delta_1,\ldots ,\delta_r)\in\mathbb{N}^r$. This condition is equivalent to saying that there exists a finite morphism $\pi : (\mathcal{X},0)\rightarrow (\mathbb{C}^r,0)$ such that the discriminant locus is contained in a normal crossing divisor.

Although q.o.h. are generalizations of plane curves, such hypersurfaces do not have isolated singularities in general.
The only possible quasi-ordinary hypersurface isolated singularities are plane curves and normal surfaces. On the other hand plane branches and irreducible q.o.h. share some important properties: in both cases, we can obtain parameterizations and determine the topological class using a finite number of certain exponents. In fact, by the Abhyankar-Jung theorem (see \cite{Ab}), if $f\in \mathbb{C}\{\underline{X} \} \left[ X_{r+1} \right]$ is an irreducible Weierstrass polynomial of degree $n$ and it defines a q.o.h. then any root $\xi$ of $f$ (called a \textit{quasi-ordinary branch}) belongs to $\mathbb{C}\left\lbrace
\underline{X}^{\frac{1}{n}} \right\rbrace:=\mathbb{C}\left\lbrace
X_1^{\frac{1}{n}},\ldots, X_r^{\frac{1}{n}} \right\rbrace$. In this way, denoting $\{\xi_{1},\ldots ,\xi_{n}\}$ the set of roots of $f$ we have
$$
\Delta_{X_{r+1}}f=\prod_{i \neq j}(\xi_{i}-\xi_{j})=\underline{X}^{\delta}\cdot u(\underline{X}) \in \mathbb{C}\left\lbrace \underline{X}\right\rbrace\ \ \mbox{with}\ u(\underline{0})\neq 0.
$$

In particular,
$ \xi_{i}-\xi_{j}=X_{1}^{\frac{\lambda_{1}(i,j)}{n}}
\cdot\ldots\cdot
X_{r}^{\frac{\lambda_{r}(i,j)}{n}}u_{ij} \in \mathbb{C}\left\lbrace
\underline{X}^{\frac{1}{n}} \right\rbrace$ with
$u_{ij}$ a unit and $\lambda_{k}(i,j) \in
\mathbb{N}$ for $k=1, \ldots, r$.

Considering the usual product order $\preceq$ in $\mathbb{N}^r$, that is, $\alpha\preceq \beta$ if and only if $\alpha_i\leq\beta_i$ for all $1\leq i\leq r$ and denoting $\lambda_{1},\ldots ,\lambda_g$ the distinct $r$-tuples $\lambda(i,j)=(\lambda_{1}(i,j),\ldots ,\lambda_{r}(i,j))$ we can reindex them in such a way that $\lambda_{1}\prec\ldots\prec\lambda_g$ (Lemma 5.6, \cite{lipman-top}). The elements $\lambda_i$ with $1\leq i\leq g$ are called \textit{(generalized) characteristic exponents} of $f$.

The generalized characteristic exponents play an important role in the topological classification of irreducible q.o.h.. Two q.o.h. $(\mathcal{X},0)$ and $(\mathcal{Y},0)$ in $\mathbb{C}^{r+1}$ are \textit{topologically equivalent} as immersed germ, if there are neighborhoods $U$ and $V$ of origin and a (germ of) homeomorphism $\tilde{\Phi}: (\mathbb{C}^{r+1},0)\rightarrow (\mathbb{C}^{r+1},0)$ such that $\tilde{\Phi}(\mathcal{X}\cap U)=\mathcal{Y}\cap V$. If $\tilde{\Phi}$ is an analytic isomorphism, so we say that $(\mathcal{X},0)$ and $(\mathcal{Y},0)$
are \textit{analytically equivalent}. As a q.o.h. $(\mathcal{X},0)\subset \mathbb{C}^{r+1}$ can be defined by distinct Weierstrass polynomials the generalized characteristic exponents are not uniquely determined. On the other hand Lipman and Gau (see \cite{lipman-top} and \cite{gau}) characterized the topological type of irreducible h.q.o. by means $\{n,\lambda_1,\ldots ,\lambda_g\}$ associated to a particular quasi-ordinary branch that they called normalized (see Section 2). In particular, we can conclude that the multiplicity of an irreducible h.q.o. is a topological invariant and it is equal to $min\{n,\sum_{i=1}^{r}\lambda_{1i}\}$ where $\lambda_1=(\lambda_{11},\ldots ,\lambda_{1r})$.

Similarly to the irreducible plane curve, we can consider a semigroup $\Gamma\subset\mathbb{N}^r$ that determines and it is determined by the set $\{n,\lambda_1,\ldots ,\lambda_g\}$, that is, $\Gamma$ is also a complete topological invariant. The analytical invariance of $\Gamma$ was shown by
Popescu-Pampu in \cite{popescu} and Gonz\'alez-P\'erez in \cite{pedro}.

As mentioned above, several properties and results about plane branches can be generalized by properly introducing concepts in the context of q.o.h.. For the convenience of the reader, we recall some of these results in Section 2.

In this paper, we explore the notion of {\it generalized Zariski exponents} that extend the Zariski invariant for irreducible plane curve introduced in \cite{zariski-inv}. Such notion was considered {\it en passant} by Panek in her thesis \cite{TeseNayene} under the supervision of the second author. In Section 3 we show that the generalized Zariski exponents are invariant for $\tilde{\mathcal{A}}$-equivalence of an irreducible normalized q.o.h. where $\tilde{\mathcal{A}}$ is a subgroup of the well-known group $\mathcal{A}$ (see Definition \ref{A-tilde}) considered in \cite{Marcelo-Nayene}.

Taking into account the generalized Zariski exponents, we explore irreducible q.o.h. such that its topological class admits a countable number of distinct $\tilde{\mathcal{A}}$-classes that we call {\it quasi-simple} singularity. For plane branches, quasi-simple are precisely simple singularities and they were classified by Bruce and Gaffney (see \cite{gaffney} or \cite{handbook}). In Section 4, we present normal forms for quasi-ordinary surfaces with one generalized characteristic exponent that are quasi-simple concerning the formal $\tilde{\mathcal{A}}$-action.

\section{Quasi-ordinary parameterizations}

In this section, we recall some results related to quasi-ordinary branches, their topological and analytical aspects.

In \cite{lipman-surfaces} (see Proposition 1.3) Lipman shows that a nonunit $\xi =\sum c_{\delta}\underline{X}^{\frac{\delta}{n}}\in\mathbb{C}\left\lbrace \underline{X}^{\frac{1}{n}}\right\rbrace$ such that $n$ and all entries of $\delta$ with $c_{\delta}\neq 0$ do not have a common nontrivial divisor is a quasi-ordinary branch if and only if there exist $r$-tuples $\lambda_{1} \prec \lambda_{2} \prec \cdots \prec \lambda_{g}$ with $c_{\lambda_{i}}\neq 0$ satisfying $\lambda_{j}\not\in Q_{j-1}:=n\mathbb{Z}^{r}+\sum_{i=1}^{j-1}\lambda_{i}\mathbb{Z}$ and
$\delta\in n\mathbb{Z}^{r}+\sum_{\lambda_{i}\preceq \delta}\lambda_{i}\mathbb{Z}$ for any $c_{\delta} \neq 0$.
The $r$-tuples $\lambda_i$ with $1\leq i\leq g$ are the {\it generalized characteristic exponents} of $\xi$.

Given a quasi-ordinary branch $\xi =\sum c_{\delta}\underline{X}^{\frac{\delta}{n}} \in \mathbb{C}\left\lbrace \underline{X}^{\frac{1}{n}}\right\rbrace$ we denote $$t_i = X_{i}^{\frac{1}{n}},\ 1\leq i\leq r\ \ \mbox{and}\ \ S(\underline{t}) = \sum c_{\delta}\underline{t}^{\delta}\in\mathbb{C}\{\underline{t}\}:=\mathbb{C}\{t_1,\ldots ,t_r\}.$$
This allow us to define a $\mathbb{C}$-algebra homomorphism
\begin{equation}\label{H*}\begin{array}{cccc}
		H^*: &	\mathbb{C}\{X_1,\ldots, X_{r+1}\} &\longrightarrow &  \mathbb{C}\{\underline{t}\} \\  &
		h(X_1,\ldots, X_{r+1}) &\mapsto& h(t_1^{n},\ldots, t_r^{n}, S(\underline{t}))
\end{array}\end{equation}
with $\frac{\mathbb{C}\{\underline{X}, X_{r+1}\}}{\langle f\rangle}\cong Im(H^*)=\mathbb{C}\{t_1^{n},\ldots,t _r^{n}, S(\underline{t}) \}$ where $f\in\mathbb{C}\{\underline{X}\}[X_{r+1}]$ is the minimal polynomial of $\xi$.

We call $H=H_{f}=(t_1^{n},\ldots, t_r^{n}, S(\underline{t}))$ a {\it quasi-ordinary parameterization} of $f$ or $\xi$.	

As we have mentioned in the introduction, a q.o.h. can be defined by distinct Weierstrass polynomials and consequently, we can obtain distinct quasi-ordinary parameterizations and generalized characteristic exponents as well. However, Lipman (see \cite{lipman-surfaces} and \cite{lipman-top}) proves that any irreducible q.o.h. admits a {\it normalized} quasi-ordinary parameterization, that is, a quasi-ordinary parameterization $H=(t_1^{n},\ldots, t_r^{n},S(\underline{t}))$ with $S(\underline{t})=\sum c_{\delta}\underline{t}^{\delta}$ satisfying
\begin{enumerate}
	\item  $S(\underline{t})= u(\underline{t})\cdot \underline{t}^{\lambda_{1}}$ with $u(\underline{0}) = 1$;
	\item $\lambda^{i}:=(\lambda_{1i},\ldots, \lambda_{gi})\geq_{lex}(\lambda_{1j},\ldots, \lambda_{gj})=:\lambda^{j}$ for all $1\leq i<j\leq r$;
	\item if $\lambda_{1} = (\lambda_{11}, 0,\ldots, 0)$, then $\lambda_{11} > n$,
\end{enumerate}
and any normalized quasi-ordinary parameterization of a q.o.h. $(\mathcal{X},0)$ admits the same generalized characteristic exponents. In addition, the topological type of an irreducible q.o.h. determines and it is determined by $\{n,\lambda_1,\ldots ,\lambda_g\}$ or equivalently by its  associated semigroup $\Gamma_H$. To describe $\Gamma_H$ we present some notions related to elements in $\mathbb{C}\{\underline{t}\}$.

\begin{definition}\label{dominant-vertices}
	Given $q=\sum c_{\alpha}\underline{t}^{\alpha}\in\mathbb{C}\{\underline{t}\}\setminus\{0\}$ and considering $supp(q)=\{\alpha;\ c_{\alpha}\neq 0\}\subseteq\mathbb{N}^r$ we denote by
	$\mathcal{N}(q)$ its Newton polyhedra, that is, the convex closure of the set $supp(q)+\mathbb{R}^{r}_{+}$ in $\mathbb{R}^{r}$. We indicate by $V_\mathcal{N}(q)$ the set of vertices of $\mathcal{N}(q)$.
	
	We say that $q\in\mathbb{C}\{\underline{t}\}\setminus\{0\}$ has \emph{dominant exponent} $\mathcal{V}(q)\in\mathbb{N}^r$ if $
	V_\mathcal{N}(q)=\{\mathcal{V}(q)\}$. Given $A\subseteq\mathbb{C}\{\underline{t}\}$ we write $\mathcal{V}(A):=\{\mathcal{V}(q);\  q\in A\setminus\{0\}\ \mbox{admits dominant exponent}\}$.
\end{definition}

Gonz\'alez-P\'erez in \cite{pedro} shows that
$$\Gamma_{H}:=\mathcal{V}(Im(H^*))=\{V_{\mathcal{N}}(H^*(h));\ h\in\mathbb{C}\{\underline{X},X_{r+1}\} \}\subseteq\mathbb{N}^r$$
is an additive semigroup generated by $\nu_k$, $k=1,\ldots ,r+g$ done by
\begin{equation}\label{generators}
	\nu_{j} = n\theta_{j},\ 1\leq j\leq r,\ \ \ \ \nu_{r+1} = \lambda_{1},\ \ \ \ \nu_{r+i} = n_{i-1}\nu_{r+i-1} + \lambda_{i} - \lambda_{i-1},\ 2\leq i\leq g\end{equation}
where $\{\theta_{j}, \; 1\leq j\leq r \}$ is the canonical $\mathbb{R}$-basis of $\mathbb{R}^r$ and $n_i=|Q_{i}:Q_{i-1}|$, $1\leq i\leq g$, that is, the index of the subgroup $Q_{i-1}=n\mathbb{Z}^{r}+\sum_{j=1}^{i-1}\lambda_{j}\mathbb{Z}$ in $Q_i=Q_{i-1}+\lambda_{i}\mathbb{Z}$. In particular, by (\ref{generators}) we conclude that $n\mathbb{Z}^r+\sum_{j=1}^{i}\nu_{r+j}\mathbb{Z}=Q_{i}=n\mathbb{Z}^r+\sum_{j=1}^{i}\lambda_{j}\mathbb{Z}$. In addition we get $n_1\cdot\ldots\cdot n_g=|Q_1:Q_0|\cdot\ldots\cdot |Q_{g}:Q_{g-1}|=n$.

The semigroup $\Gamma_{H}$ is called the {\it associated semigroup} of $H$.

\begin{remark}\label{std-rep}
	Given $\gamma \in Q_{k}$ for some $k \in \{0,\ldots, g\}$  there are unique $a_1,\ldots, a_{r+k} \in \mathbb {Z}$ with $0 \leq a_{r+j} < n_j$ for $j = 1,\ldots, k$, such that $\gamma=\sum_{i=1}^{r+k }a_{i}\nu_{i}$. Such a representation of $\gamma$ is called \emph{standard representation}. Moreover, if $\gamma=\sum_{i=1}^{r+k}a_{i}\nu_{i} \in Q_{k}$ is given by a standard representation, we have  $ \gamma\in \Gamma_{k} :=\langle \nu_{1},\ldots,\nu_{r+k}\rangle=\sum_{i=1}^{r+k}\nu_i\mathbb{N}$ if and only if $a_i \geq 0$ for all $1 \leq i \leq r$ (see \cite{assi}).
\end{remark}

A finner equivalence relation is the analytic equivalence. Given two q.o.h. $(\mathcal{X}_1,0), (\mathcal{X}_2,0)\subset (\mathbb{C}^{r+1},0)$, defined by Weierstrass polynomials $f_1, f_2\in\mathbb{C}\{\underline{X}\}[X_{r+1}]$, we say that they are {\it analytically equivalent} if and only if $$\mathbb{C}\{t_1^n,\ldots ,t_r^n ,S_1(\underline{t})\}\cong\frac{\mathbb{C}\{\underline{X},X_{r+1}\}}{\langle f_1\rangle}\cong \frac{\mathbb{C}\{\underline{X},X_{r+1}\}}{\langle f_2\rangle}\cong \mathbb{C}\{t_1^n,\ldots ,t_r^n ,S_2(\underline{t})\}$$ as $\mathbb{C}$-algebras, where $H_i=(t_1^n,\ldots ,t_r^n ,S_i(\underline{t}))$ is a quasi-ordinary parameterization (not necessarily normalized) of $(\mathcal{X}_i,0)$ for $i=1,2$. Considering the group $\mathcal{A}=\mbox{Iso}(\mathbb{C}^{r+1},0)\times \mbox{Iso}(\mathbb{C}^{r},0)$ where $\mbox{Iso}(\mathbb{C}^k,0)$ denotes the group
of analytic isomorphisms of $(\mathbb{C}^{k},0)$ and identifying $H_i$ with a map germ $H_i:(\mathbb{C}^r,0)\rightarrow (\mathbb{C}^{r+1},0)$, the analytical equivalence of q.o.h. can be expressed by the $\mathcal{A}$-{\it equivalence} of $H_1$ and $H_2$, that is, $H_2 =\sigma \circ H_1 \circ \rho^{-1}$ for some $(\sigma, \rho)\in\mathcal{A}$ that we denote $H_1\underset{\mathcal{A}}{\sim} H_2$.

In \cite{Marcelo-Nayene} the second author and Panek consider a subgroup $\widetilde{\mathcal{A}}$ of $\mathcal{A}$ to detect terms in a quasi-ordinary parameterization that can be eliminable by the $\widetilde{\mathcal{A}}$-action (see Proposition 2.15, Theorem 3.1 and Corollary 3.6, \cite{Marcelo-Nayene}). In this work, we introduce an invariant concerning $\widetilde{\mathcal{A}}$-action and we explore it to classify quasi-ordinary hypersurfaces in a given topological class.

For the convenience of the reader, we present the description of the group $\widetilde{\mathcal{A}}$ and some results concerning it.

\begin{definition}\label{A-tilde} Fix a topological class of an irreducible q.o.h. in $(\mathbb{C}^{r+1},\underline{0})$ determined by $\{n,\lambda_1,\ldots ,\lambda_g\}$. We denote by $\widetilde{\mathcal{A}}$ the subgroup of $\mathcal{A}$
	consisting of elements $(\sigma,\rho)\in\mathcal{A}$ with
	$\rho=(t_1(c_1+\zeta_1),\ldots ,t_r(c_r+\zeta_r))$ and
	$\sigma(\underline{X},X_{r+1})=(\sigma_1(\underline{X},X_{r+1}),\ldots
	,\sigma_{r+1}(\underline{X},X_{r+1}))$ such that
	$$\sigma_i(\underline{X},X_{r+1})=a_iX_i+P_i,\
	\ \ \ \sigma_{r+1}(\underline{X},X_{r+1})=
	a_{r+1}X_{r+1}+P_{r+1},$$
	\begin{equation}\label{cond-mud}
		\begin{array}{c}
			P_i=X_i\epsilon_i+X_{r+1}\eta_i,\ \ \ P_{r+1}=X_{r+1}\epsilon_{r+1}+\underline{X}^{\alpha}\eta_{r+1},\ \ \ \alpha=\left (\left \lceil \frac{\lambda_{11}}{n}\right \rceil,\ldots ,\left \lceil \frac{\lambda_{1r}}{n}\right \rceil\right )\ \ \mbox{and} \\
			c_i,a_i,a_{r+1}\in\mathbb{C}\setminus\{0\},\ \ \ \zeta_i\in \mathcal{M}_r,\ \ \
			\epsilon_i,\epsilon_{r+1}\in\mathcal{M}_{r+1};\\ \eta_i,\eta_{r+1}\in\mathbb{C}\{\underline{X},X_{r+1}\};\ \
			\eta_{i}=0\ \mbox{if}\ n>\lambda_{1i}
			\ \ \mbox{for}\ \ i=1,\ldots ,r, 
		\end{array}
	\end{equation}
	where $\mathcal{M}_r=\langle \underline{t}\rangle$ and
	$\mathcal{M}_{r+1}=\langle \underline{X},X_{r+1}\rangle$ denote the
	maximal ideals of $\mathbb{C}\{\underline{t}\}$ and
	$\mathbb{C}\{\underline{X},X_{r+1}\}$, respectively.	
\end{definition}

If $H_1=(t_1^n,\ldots ,t_r^n ,S_1(\underline{t}))$ is a q.o.  parameterization, then $(h_1,\ldots ,h_{r+1})=\sigma\circ H_1\circ\rho^{-1}$ is not a quasi-ordinary parameterization for any $(\sigma,\rho)\in\widetilde{\mathcal{A}}$. In fact, considering  $(\sigma,\rho)\in\widetilde{\mathcal{A}}$ as described in the above definition and denoting $\rho^{-1}=((\rho^{-1})_1,\ldots ,(\rho^{-1})_r)$, in order to obtain $h_i=t_i^n$ for $1\leq i\leq r$
we must have
$$h_i=\left ((\rho^{-1})_i(\underline{t})\right )^n.\left (a_{i}+\epsilon_{i}\circ H_1\circ\rho^{-1}(\underline{t})\right )+S_1(\rho^{-1}(\underline{t}))\cdot(\eta_i\circ H_1\circ\rho^{-1}(\underline{t}))=t_i^n=(\rho_i\circ\rho^{-1}(\underline{t}))^{n},$$
and consequently
$$
\rho_i(\underline{t})=\left ( t_i^n\cdot (a_i+\epsilon_{i}\circ H_1)+S_1(\underline{t})\cdot(\eta_{i}\circ H_1)\right )^{\frac{1}{n}}=t_i\cdot\left (a_i+\frac{H_1^*(X_i\epsilon_{i}+X_{r+1}\eta_{i})}{t_i^n}\right )^{\frac{1}{n}}.
$$
Recall that $\eta_{i}=0$ if $n>\lambda_{1i}$ (see (\ref{generators})).

The above computations show us that the elements $(\sigma, \rho)\in\widetilde{\mathcal{A}}$ for which we have that $H_2=\sigma\circ H_1\circ\rho^{-1}$ is a q.o. parameterization are
\begin{equation}\label{mud-preserva}
	\begin{array}{c}
		\sigma_i(\underline{X},X_{r+1})=a_iX_i+P_i, \ \ \ \ \sigma_{r+1}(\underline{X},X_{r+1})=
		a_{r+1}X_{r+1}+P_{r+1},\vspace{0.2cm}\\
		\rho_i(\underline{t})=t_i\cdot\left (a_i+\frac{H_1^*(P_i)}{t_i^n}\right )^{\frac{1}{n}},
	\end{array}
\end{equation}
with $a_i, a_{r+1}\in\mathbb{C}\setminus\{0\}$, $P_{i}$ and $P_{r+1}$ satisfying the conditions (\ref{cond-mud}) for $1\leq i\leq r$. 

Notice that the above elements are well behaved concerning the finite morphism $\pi: (\mathcal{X},0)\rightarrow (\mathbb{C}^r,0)$ and they are similar to the change of coordinates presented by Gonz\'alez-P\'erez (see Lemma 2.7 in \cite{pedro-tese}). However, the change of coordinates given in (\ref{mud-preserva})
depends on $H_1$ and therefore the set of such elements is not a subgroup of $\mathcal{A}$ (neither $\tilde{\mathcal{A}}$). 

The reader is invited to compare the above description with Proposition 1.2.3 of \cite{handbook} which addresses the particular case of plane curves.

\begin{definition}\label{H-A1}
	In what follows we denote
	$$\begin{array}{cc}
		\mathcal{H} = \left\{
		(\sigma, \rho) \in \tilde{\mathcal{A}};\ \sigma(\underline{X},X_{r+1}) = (a_1X_1,\ldots, a_{r+1}X_{r+1}),
		\rho(\underline{t}) = (c_1t_1,\ldots, c_rt_r),\ a_i \neq 0 \neq c_i \right\}\vspace{0.2cm}\\
		\widetilde{\mathcal{A}}_1=\left \{ 
		(\sigma, \rho) \in \tilde{\mathcal{A}};\ j^1\sigma(\underline{X},X_{r+1}) = (X_1+d_1X_{r+1},\ldots, X_r+d_rX_{r+1},X_{r+1})\ \mbox{and}\
		j^1\rho(\underline{t}) = (\underline{t})
		\right \},
	\end{array}
	$$
	where $d_i\in\mathbb{C}$ ($d_i=0$ if $n>\lambda_{1i}$) and $j^kh$ is the $k$th jet of $h$.
\end{definition}

Notice that any element $(\sigma, \rho)\in\widetilde{\mathcal{A}}$ can be express as a composition of an element $(\sigma_0,\rho_0)\in\mathcal{H}$ and an element $(\sigma_1,\rho_1)\in\widetilde{\mathcal{A}}_1$. 

Given a q.o. parameterization $H = \left(t^n_1,\cdots, t^n_r, S(\underline{t})\right)$ where $S(\underline{t})=\sum_{\delta\succeq\lambda_{1}}b_{\delta}\underline{t}^{\delta}$ we say that $\underline{t}^{\gamma}$ with $\gamma\in supp(S(\underline{t}))$ is {\it eliminable} (by $\tilde{\mathcal{A}}$-action) if there exists a q.o. parameterization $H'=(t_1^n,\ldots ,t_r^n,\sum_{\delta\succeq\lambda_{1}}c_{\delta}\underline{t}^{\delta})$ with $H'$ $\tilde{\mathcal{A}}$-equivalent to $H$ such that $c_{\gamma}=0$ and $b_{\delta}=c_{\delta}$, except eventually for $\delta\succ\gamma$.

If $H_1=(t_1^n,\ldots ,t_r^n,S_1(\underline{t}))$ is a q.o. parameterization and $(\sigma,\rho)\in\mathcal{H}$ satifies (\ref{mud-preserva}), that is, $c_i=a_i$ for $1\leq i\leq r$, then $(t_1^n,\ldots ,t_r^n,S_2(\underline{t}))=:H_2=\sigma\circ H_1\circ\rho^{-1}(\underline{t})=(t_1^n,\ldots ,t_r^n,S_1(a_1^{-1}t_1,\ldots ,a_r^{-1}t_r))$ in particular, $supp(S_2(\underline{t}))=supp(S_1(\underline{t}))$. So, elements in $\mathcal{H}$ do not introduce or eliminate terms in a q.o. parameterization but allow us to normalize some coefficients as described in the following proposition.

\begin{proposition}[Proposition 2.5 in \cite{Marcelo-Nayene}]\label{normalize}
	Let $H = \left(t^n_1,\ldots, t^n_r, S(\underline{t})\right)$ be a q.o. parameterization. If $P\subseteq \{\zeta-\lambda_{1};\ \lambda_{1}\neq\zeta\in supp(S(\underline{t}))\}\subset\mathbb{R}^{r}$ is a linearly independent set, then there exists $(\sigma, \rho) \in \mathcal{H}$ such that all terms with exponent belonging to $\{\delta+\lambda_1;\ \delta\in P\cup\{\underline{0}\}\}$ in $\sigma\circ H\circ\rho^{-1}$ are monic.
\end{proposition}

Let $\Omega^r:=\Omega^r_{\mathbb{C}\{\underline{X},X_{r+1}\}}=\sum_{i=1}^{r+1}\mathbb{C}\{\underline{X},X_{r+1}\}dX_{1}\wedge\cdots\wedge\widehat{dX_{i}}\wedge\cdots\wedge dX_{r+1}$ be the $\mathbb{C}\{\underline{X},X_{r+1}\}$-module of differential $r$-forms and the map
\begin{eqnarray}\label{Psi}
	\Psi_H : \Omega^{r}&\rightarrow& \mathbb{C}\{\underline{t}\}\\ \nonumber
	\omega & \mapsto &  \frac{t_1\cdots t_r}{dt_1\wedge \cdots \wedge dt_r}\left (\sum_{i=1}^{r+1}H^{*}(h_i)\bigwedge_{j=1;j\neq i}^{r+1}dH^{*}(X_j)\right )
\end{eqnarray}
where $\omega=\sum_{i=1}^{r+1}h_{i}dX_{1}\wedge\cdots\wedge\widehat{dX_{i}}\wedge\cdots\wedge dX_{r+1}$ and $H^*$ is the homomorphism given in (\ref{H*}).

Similarly to the semigroup $\Gamma_H$ and considering the notion of dominant exponent given in Definition \ref{dominant-vertices} we put
$$
	\Lambda_H:=\mathcal{V}(Im(\Psi_H)).
$$

The set $\Lambda_H$ was introduced in \cite{Marcelo-Nayene} considering the $\frac{\mathbb{C}\{\underline{X},X_{r+1}\}}{\langle f\rangle}$-module K\"ahler $r$-forms for a quasi-ordinary hypersurface defined by $f$ with parameterization $H$. It follows that $\Lambda_{H}$ is a $\Gamma_{H}$-monomodule, that is, $\Gamma_{H}+\Lambda_{H}\subset \Lambda_{H}$. 

We can use $\Lambda_{H}$ to identify terms in a quasi ordinary parameterization that are amenable to elimination through coordinate changes as given in Definition \ref{A-tilde}, that is, considering the $\tilde{\mathcal{A}}$-action. More specifically we highlight:

\begin{proposition}[Corollary 3.6 in \cite{Marcelo-Nayene}]\label{elimination}
	Let $H = \left(t^n_1,\cdots, t^n_r, S(\underline{t})\right)$ be a q.o. parameterization.
	If $\gamma= \mathcal{V}(\Psi_H(\omega))\in \Lambda_{H}$ with $\gamma_{i}\geq n$ and $\omega=\sum_{i=1}^{r+1}(-1)^{r+1-i}P_idX_{1}\wedge\cdots\wedge\widehat{dX_{i}}\wedge \cdots\wedge dX_{r+1} \in \Omega^{r}$ where $P_i$ is as described in (\ref{cond-mud}) for all $i=1,\ldots,r$, then  $\underline{t}^{\gamma-\underline{n}}$ is eliminable (by $\tilde{\mathcal{A}}_1$-action).
\end{proposition}

In the next section, we present an $\tilde{\mathcal{A}}$-invariant that can be easily computed using a q.o. parameterization.

\section {Generalized Zariski exponents}

Considering an irreducible plane curve $C$ with parameterization $(t^n,\sum_{i\geq \lambda_{1}}a_it^i)$ and associated semigroup $\Gamma=\langle v_0:=n,v_1:=\lambda_{1},v_2,\ldots ,v_g\rangle$, Zariski (see \cite{zariski-inv}, \cite{zariski-book} or \cite{handbook}) shows that there exists a curve  analytically equivalent to $C$ with parameterization 
$$
	\left (t^n,t^{\lambda_{1}}+\sum_{i>\lambda_{1}}a'_it^i\right )\ \mbox{with}\ i\not\in\Gamma\bigcup\{\mathbb{N}\lambda_{1}+2\lambda_1-n\} 
$$
that he calls a {\it short parameterization}. 
In addition, given a short parameterization as before, Zariski shows that if $supp\left ( \sum_{i>\lambda_{1}} a'_it^i\right )\neq\emptyset$ then the minimum element of this set is an analytical invariant called {\it Zariski exponent} or {\it Zariski invariant}. Notice that $supp\left ( \sum_{i>\lambda_{1}} a'_it^i\right )=\emptyset$ if and only if the plane curve $C$ is analytically equivalent to a curve with parameterization $(t^{n},t^{\lambda_{1}})$, that is, it is defined by the quasi-homogeneous polynomial $Y^{n}-X^{\lambda_{1}}$.

In this section, we propose a generalization of the concepts of short parameterization and Zariski exponent for quasi-ordinary hypersurfaces.

\begin{proposition}\label{normalize-prop}
	Let $H=(t_1^n,\ldots ,t_r^n,S(\underline{t}))$  be a normalized q.o. parameterization with $S(\underline{t})=t^{\lambda_{1}}+\sum_{\delta\succ\lambda_{1}}b_{\delta}\underline{t}^{\delta}$ and associated semigroup $\Gamma$. Given $\gamma\in supp\left (\sum_{\delta\succ\lambda_{1}}b_{\delta}\underline{t}^{\delta} \right )$ if
	\begin{equation}\label{exp-eliminaveis}
		\gamma \in \Gamma \bigcup_{1\leq i\leq r\atop \lambda_{1i}\geq n}\left( \Gamma+2\lambda_1-\nu_i \right )
	\end{equation}
	then $\underline{t}^{\gamma}$ is eliminable by $\tilde{\mathcal{A}}_1$-action.
\end{proposition}
\begin{proof}
	Let us consider $\omega=P_{r+1}dX_1\wedge\cdots\wedge dX_r$ with $\mathcal{V}(P_{r+1})=\gamma\in\Gamma$. Notice that $\gamma\succ\lambda_{1}$ implies that $P_{r+1}$ satisfies (\ref{cond-mud}). As $\Psi_H(\omega)=n^r\cdot H^*(P_{r+1})\cdot \underline{t}^{\underline{n}}$, we get $\mathcal{V}(\Psi_H(\omega))=\gamma+\underline{n}$ and, by Proposition \ref{elimination}, the term $\underline{t}^{\gamma}$ is eliminable  by $\tilde{\mathcal{A}}_1$-action.
	
	On the other hand, if $\lambda_{1i}\geq n$ and $\gamma=\delta+2\lambda_{1}-\nu_i$ with $\delta\in \Gamma$, we take $$\omega = P_idX_1\wedge\cdots \wedge\widehat{dX_{i}}\wedge\cdots\wedge dX_{r}\wedge dX_{r+1}\ \ \mbox{with}\ \  
	P_i=X_{r+1}\eta_i\ \ \mbox{and}\ \ \mathcal{V}(\eta_{i})=\delta.$$ In this way, 
	$$\Psi_H(\omega)= n^{r-1}(-1)^{r-i}\cdot H^*(X_{r+1}\eta_i)\cdot\underline{t}^{\underline{n}-\nu_i}\cdot\left ( \lambda_{1i}\underline{t}^{\lambda_{1}}+\sum_{\delta\succ\lambda_{1}}\delta_ib_{\delta}\underline{t}^{\delta}\right ),$$
	that is, $\gamma=\mathcal{V}(\Psi_H(\omega))-\underline{n}$ and, by Proposition \ref{elimination}, the term $\underline{t}^{\gamma}$ is eliminable  by $\tilde{\mathcal{A}}_1$-action.
\end{proof}

By the previous proposition, given a normalized q.o. parameterization $H$ we can consider (possibly infinitely many) changes of coordinates and we obtain $H'=(t_{1}^{n}, \ldots, t_{r}^{n}, S(\underline{t}))$ such that any exponent in $S(\underline{t})-\underline{t}^{\lambda_1}$ does not belong to the union of sets (\ref{exp-eliminaveis}). We avoid verifying if a composition of infinite many of elements of $\mathcal{A}$ is analytic, for this reason, when we consider eventually infinite many changes of coordinates to obtain $H'$ from $H$, we say that $H'$ is {\it formally $\mathcal{A}$-equivalent} to $H$.

\begin{example}\label{n=2}
	Let us consider a normalized q.o. parameterization $H=(t_1^2,\ldots ,t_r^2,\underline{t}^{\lambda_1}+\sum_{\gamma\succ\lambda_1}a_{\gamma}\underline{t}^{\gamma})$. As $2=|Q_g:Q_0|=n_1\cdot\ldots\cdot n_g$, we must have $g=1$, that is, the associated semigroup of $H$ is $\Gamma=\langle \nu_1,\ldots ,\nu_r,\nu_{r+1}\rangle$. By Remark \ref{std-rep}, we can write any $\gamma\in Q_1$ as $\gamma=2(a_1,\ldots ,a_r)+a\cdot\nu_{r+1}$ with $0\leq a<n_1=n=2$ and $a_i\in\mathbb{Z}$ for $1\leq i\leq r$. In this way, if $\gamma\succ\lambda_1=\nu_{r+1}$ we must have $a_i\geq 0$ for every $1\leq i\leq r$, that is, $\gamma\in\Gamma$ and consequently, by Proposition \ref{normalize-prop}, $H$ is formally $\mathcal{A}$-equivalent to $(t_1^2,\ldots ,t_r^2,\underline{t}^{\lambda_1})$.
\end{example}

\begin{example}\label{normal}
	If a q.o. parameterization $H$ has $\lambda_{1}=\underline{1}=(1,\ldots ,1)$ and $n>1$ then, as $n_1\lambda_1\in Q_0=n\mathbb{Z}^r$, we must have $n_1=|Q_1:Q_0|=n=n_1\cdot\ldots\cdot n_g$, that is $g=1$ and the value semigroup of $H$ is given by $\Gamma=\langle \nu_1,\ldots ,\nu_r,\nu_{r+1}\rangle$. By Remark \ref{std-rep}, any $\gamma\in Q_1$ can be expressed as $\gamma=n(a_1,\ldots ,a_r)+a\cdot\nu_{r+1}$ with $0\leq a<n$ and $a_i\in\mathbb{Z}$ for $1\leq i\leq r$. In order to $\gamma\succ\lambda_{1}=\nu_{r+1}$ we get $a_i\geq 0$ for all $1\leq i\leq r$, that is, $\gamma\in\Gamma$. In this way, by Proposition \ref{normalize-prop}, $H$ is formally $\mathcal{A}$-equivalent to $(t_1^n,\ldots , t_r^n,t_1\cdot\ldots \cdot t_r)$.  
\end{example}

The Propositon \ref{normalize-prop} motivates the following definition.

\begin{definition}\label{zariski-exponents}
	Let $H=(t_{1}^{n}, \ldots, t_{r}^{n}, t^{\lambda_{1}}+\sum_{\delta\succ\lambda_{1}}a_{\delta}\underline{t}^{\delta})$ be a normalized q.o. parameterization with value semigroup $\Gamma$. We say that $H$ is a \emph{quasi-short parameterization} if every element in $supp(\sum_{\gamma\succ\lambda_{1}}a_{\delta}\underline{t}^{\delta})$ does not belong to $\Gamma \bigcup_{1\leq i\leq r\atop \lambda_{1i}\geq n}\left( \Gamma+2\lambda_1-\nu_i \right )$. 
	
	If $H=(t_{1}^{n}, \ldots, t_{r}^{n},t^{\lambda_{1}}+\sum_{\delta\succ\lambda_{1}}a_{\delta}\underline{t}^{\delta})$ is a quasi-short parameterization and $\sum_{\delta\succ\lambda_{1}}a_{\delta}\underline{t}^{\delta}\neq 0$ the we call \emph{Generalized Zariski exponents} the elements in
	$E_{\mathcal{Z}}(H):=\min_{\preceq} supp(\sum_{\delta\succ\lambda_{1}}a_{\delta}\underline{t}^{\delta})$. If $\sum_{\delta\succ\lambda_{1}}a_{\delta}\underline{t}^{\delta}=0$ we put $E_{\mathcal{Z}}(H)=\{\underline{\infty}\}$.
\end{definition}

It follows that any q.o. parameterization with generalized Zariski exponents $E_{\mathcal{Z}}$ is formally $\mathcal{A}$-equivalent to a quasi-short parameterization 
$$
H=\left (t_{1}^{n}, \ldots, t_{r}^{n}, \underline{t}^{\lambda_{1}}\right ) \ \ \mbox{if}\ \ E_{\mathcal{Z}}=\{\underline{\infty}\}\ \ \ \ \ \mbox{or}\ \ \ \ \ H=\left (t_{1}^{n}, \ldots, t_{r}^{n},\underline{t}^{\lambda_1}+\sum_{\delta\in E_{\mathcal{Z}}}a_{\delta}\underline{t}^{\delta}u_{\delta}(\underline{t})\right ) \ \ \mbox{if}\ \ E_{\mathcal{Z}}\neq\{\underline{\infty}\}
$$
where $a_{\delta}\in\mathbb{C}\setminus\{0\}$ and $u_{\delta}(\underline{0})=1$. 

\begin{remark}\label{lambda2}
	If $g\geq 2$, that is, $\lambda_2\in supp(S(\underline{t}))$ then $\lambda_2\not \in \Gamma \bigcup_{1\leq i\leq r\atop \lambda_{1i}\geq n}\left( \Gamma+2\lambda_1-\nu_i \right)$. In fact, by (\ref{generators}), we get $\lambda_2=\nu_{r+2}+\nu_{r+1}-n_1\lambda_1$, as $2\leq n_1=|Q_1:Q_0|$ it follows that $n_1\lambda_1=n\alpha$ with $\alpha=(\alpha_1,\ldots ,\alpha_r)\in(\mathbb{N}^r)^*:=\mathbb{N}^r\setminus\{\underline{0}\}$ and, by Remark \ref{std-rep}, we conclude that $\lambda_2\not\in\Gamma$. On the other hand, suppose that $\lambda_{1i}\geq n$, so $\alpha_i\geq 2$. If $\lambda_2\in\Gamma+2\lambda_1-\nu_i$, that is, $\lambda_2-\lambda_1+\nu_i\in\Gamma+\lambda_1\subset\Gamma$ then, as $\nu_{r+1}=\lambda_1$, we get
	$$\nu_{r+2}-n(\alpha_1,\ldots ,\alpha_{i-1},\alpha_i-1,\alpha_{i+1},\ldots, \alpha_r)=\nu_{r+2}-n\alpha+\nu_i=\lambda_2-\lambda_1+\nu_i\in\Gamma$$
	that, by Remark \ref{std-rep}, is an absurd. So, $\lambda_2\not \in \Gamma \bigcup_{1\leq i\leq r\atop \lambda_{1i}\geq n}\left( \Gamma+2\lambda_1-\nu_i \right)$ and for $g\geq 2$ we always have $E_{\mathcal{Z}}(H)\neq\{\underline{\infty}\}$.
\end{remark}

Notice that, for $r=1$, that is, if $H=(t^n,S(t))$ is a parameterization of a plane curve we have $\lambda_{1}>n$ and $\mathbb{N}\lambda_{1}+2\lambda_1-n\subset\Gamma+2\lambda_1-n$ so, we get the short parameterization notion introduced by Zariski.

\begin{theorem}\label{teorema}
	The generalized Zariski exponents are $\widetilde{\mathcal{A}}$-invariant.
\end{theorem}
\begin{proof}
	Let $H_{1}$ and $H_2$ be two quasi-short parameterization with same semigroup $\Gamma=\langle \nu_1,\ldots ,\nu_{r+g}\rangle$. Let us suppose that $H_1$ and $H_2$ are $\widetilde{\mathcal{A}}$-equivalent and they are given by
	$$
	H_{1}=\left (t_{1}^{n},\ldots,t_{r}^{n},S_1(\underline{t})\right ) \ \ \ \mbox{and}\ \ \  H_{2}=\left (t_{1}^{n},\ldots,t_{r}^{n},S_2(\underline{t})\right )
	$$
	with 	
	$S_1(\underline{t})=\underline{t}^{\lambda_{1}}+
	\sum_{\delta\in E} a_{\delta}\underline{t}^{\delta}u_1(\underline{t})$, $S_2(\underline{t})=\underline{t}^{\lambda_{1}}+
	\sum_{\delta\in E} b_{\delta}\underline{t}^{\delta}u_2(\underline{t})$ where
	$u_1(\underline{0})=u_2(\underline{0})=1$, $a_{\delta}\neq 0$ for every $\delta\in E$ and $E:=E_{\mathcal{Z}}(H_1)\neq\{\underline{\infty}\}$ is the set of generalized Zariski exponents of $H_1$.
	
	Considering $(\sigma,\rho) \in \widetilde{\mathcal{A}}$ such that $\sigma\circ H_{1}\circ \rho^{-1}=H_{2}$, or equivalent $\sigma\circ H_{1}=H_2\circ\rho$, we will show that $b_{\delta}\neq 0$ for every $\delta\in E$ and consequently, the generalized Zariski exponents are $\widetilde{\mathcal{A}}$ invariants.
	
	As we have remarked after the Definition \ref{H-A1} it is sufficient to consider $(\sigma,\rho)\in\tilde{\mathcal{A}}_1$, that is, 
	$$\begin{array}{c}
		\sigma_i(\underline{X},X_{r+1})=X_i+P_i(\underline{X},X_{r+1}),\ \ \ \ \sigma_{r+1}(\underline{X},X_{r+1})=
		X_{r+1}+P_{r+1}(\underline{X},X_{r+1}),\vspace{0.2cm}\\
		\rho_i(\underline{t})=t_i\cdot\left (1+\frac{H_1^*(P_i(\underline{X},X_{r+1}))}{t_i^n}\right )^{\frac{1}{n}},
	\end{array}
	$$
	with $P_{r+1}(\underline{X},X_{r+1})=X_{r+1}\cdot \epsilon_{r+1}+\underline{X}^{\alpha}\cdot \eta_{r+1}$ and $P_i(\underline{X},X_{r+1})=X_i\cdot \epsilon_i+X_{r+1}\cdot \eta_i$, $1\leq i\leq r$ satisfying the conditions (\ref{cond-mud}).
	
	Notice that 
	$$\sigma\circ H_1\circ\rho^{-1}(\underline{t})=(t_1^n,\ldots ,t_r^n,S_1\circ\rho^{-1}(\underline{t})+P_{r+1}\circ H_1\circ\rho^{-1}(\underline{t}))=(t_1^n,\ldots ,t_r^n,S_2(\underline{t}))=H_2(\underline{t}),$$
	that is, we must have \begin{equation}\label{equivalencia}S_1(\underline{t})+P_{r+1}(H_1(\underline{t}))=S_2(\rho(\underline{t})).
	\end{equation}
	
	As $P_{r+1}(\underline{X},X_{r+1})=X_{r+1}\cdot \epsilon_{r+1}+\underline{X}^{\alpha}\cdot \eta_{r+1}$ with $\epsilon_{r+1}\in\mathcal{M}_{r+1}$, $\eta_{r+1}\in\mathbb{C}\{\underline{X},X_{r+1}\}$ and $\alpha=\left (\left \lceil \frac{\lambda_{11}}{n}\right \rceil,\ldots ,\left \lceil \frac{\lambda_{1r}}{n}\right \rceil\right )$ we have that any element in $supp(P_{r+1}(H_1(\underline{t}))$ belongs to $\Gamma$ or $E+(\mathbb{N}^r)^*$, that is, all terms $a_j\underline{t}^{\delta_{j}}$ with $\delta_j\in E$ remain unchanged on left side of (\ref{equivalencia}). In addition, as $	\rho_i(\underline{t})=t_i\cdot\left (1+\frac{H_1^*(P_i(\underline{X},X_{r+1}))}{t_i^n}\right )^{\frac{1}{n}},$
	where $P_i(\underline{X},X_{r+1})=X_i\cdot \epsilon_i+X_{r+1}\cdot \eta_i$, $\epsilon_{i}\in\mathcal{M}_{r+1}$, $\eta_i\in\mathbb{C}\{\underline{X},X_{r+1}\}$ for $1\leq i\leq r$ with $\eta_i=0$ if $n>\lambda_{1i}$ we get
	$$S_2(\rho(\underline{t}))=\prod_{i=1}^{r}\rho_i^{\lambda_{1i}}+\sum_{\delta\in E}b_{\delta}\underline{t}^{\delta}v_j(\underline{t})$$
	with $v_i(\underline{0})=1$. As  $P_i(\underline{X},X_{r+1})=X_i\cdot \epsilon_i+X_{r+1}\cdot \eta_i$ we can assume that $\eta_i\not\in\langle X_i\rangle$.
	
	We will show that $supp\left ( \prod_{i=1}^{r}\rho_i^{\lambda_{1i}}\right )\cap E=\emptyset$, that is, $a_{\delta}=b_{\delta}$ for all $\delta\in E:=E_{\mathcal{Z}}(H_1)$.
	
	Using the description of $\rho_i$ for $1\leq i\leq r$ we have
$$
		\prod_{i=1}^{r}\rho_i^{\lambda_{1i}}=\underline{t}^{\lambda_{1}}\cdot\prod_{1\leq i\leq r\atop \lambda_{1i}<n}\left ( 1+H_1^*(\epsilon_{i})\right )^{\frac{\lambda_{1i}}{n}}\cdot\prod_{1\leq i\leq r\atop \lambda_{1i}\geq n}\left ( 1+H_1^*(\epsilon_{i})+\frac{H_1^{*}(X_{r+1}\eta_i)}{t_i^{n}}\right )^{\frac{\lambda_{1i}}{n}}.
$$
	
	We analyse $supp\left ( \prod_{i=1}^{r}\rho_i^{\lambda_{1i}}\right )$ considering in the above expression the expansion $(1+z(\underline{t}))^{\alpha}=\sum_{k\geq 0}\left ( \alpha\atop k\right )z^k$ for any $\alpha\in\mathbb{Q}_{> 0}$ and $z(\underline{t})\in\langle \underline{t}\rangle=\mathcal{M}_{r}\subset \mathbb{C}\{\underline{t}\}$. 
	
	If $\eta_i=0$ for every $1\leq i\leq r$, then $supp\left ( \prod_{i=1}^{r}\rho_i^{\lambda_{1i}}\right )\subset \Gamma\cup\left ( E+(\mathbb{N}^r)^*\right )$ because $\epsilon_{i}\in\mathcal{M}_{r+1}$. By definition $\Gamma\cap E=\emptyset$ so, in order to have (\ref{equivalencia}) we must have $a_{\delta}=b_{\delta}$ for every $\delta\in E$ and we can not eliminate any generalized Zariski exponent by $\widetilde{\mathcal{A}}$-action.
	
	If there exists $\eta_i\neq 0$, that can happen for $\lambda_{1i}\geq n$ then, in 
	$supp\left ( \prod_{i=1}^{r}\rho_i^{\lambda_{1i}}\right )$, besides the elements in $\Gamma\cup\left ( E+(\mathbb{N}^r)^*\right )$ we eventually obtain elements on form $\lambda_1+k(\gamma+\lambda_1-\nu_i)$ where $\gamma\in V_{\mathcal{N}}(H_1^*(\eta_i))\subset\Gamma$ for $k\geq 1$. 
	
	It is sufficient to analyze the possibility $\gamma+\lambda_{1}-\nu_i\not\in\Gamma$ with $\gamma\in\Gamma_1=\langle \nu_1,\ldots ,\nu_{r+1}\rangle$. In fact, if $\gamma+\lambda_1-\nu_i\in\Gamma$ then $\lambda_1+k(\gamma+\lambda_{1}-\nu_i)\in\Gamma$ for every $k\geq 1$, consequently $supp\left ( \prod_{i=1}^{r}\rho_i^{\lambda_{1i}}\right )\subset \Gamma$. On the other hand if $g\geq 2$ and $\gamma=\sum_{i=1}^{r+g}\alpha_i\nu_i$ with $\alpha_j\neq 0$ for some $j>r+1$ then it follows, by Remark \ref{lambda2} and the inequality $\lambda_{1i}\geq n$, that $\lambda_1+k(\gamma +\lambda_1-\nu_i)\in E+(\mathbb{N}^r)^*$. In both situations we get
	$\lambda_1+k(\gamma +\lambda_1-\nu_i)\in\Gamma\cup\left ( E+(\mathbb{N}^r)^*\right )$ and similarly to the case $\eta_{i}=0$ we conclude that $a_{\delta}=b_{\delta}$ for every $\delta \in E$.
	
	Considering $\gamma+\lambda_{1}-\nu_{i}\not\in\Gamma$ with $\gamma\in\langle \nu_1,\ldots ,\nu_{r+1}\rangle$ we may have the possibilities:
	\begin{itemize}
		\item[Case 1:]$\gamma+2\lambda_{1}-\nu_i\in E+(\mathbb{N}^r)^*$. In this case, $\lambda_{1}+k(\gamma+\lambda_{1}-\nu_i)\in E+(\mathbb{N}^r)^*$ for any $k\geq 1$ and $supp\left ( \prod_{i=1}^{r}\rho_i^{\lambda_{1i}}\right )\subset\Gamma \cup (E+(\mathbb{N}^r)^*)$ so, $a_{\delta}=b_{\delta}$ for every $\delta\in E$.
		\item[Case 2:] $\gamma+2\lambda_1-\nu_i\not\in\Gamma\cup (E+(\mathbb{N}^r)^*)$. 
		
		As $\delta\not\in \bigcup_{1\leq i\leq r\atop \lambda_{1i}\geq n}\left( \Gamma+2\lambda_1-\nu_i \right )$ for every $\delta\in E_{\mathcal{Z}}(H_1)$ and we do not have any element in $\Gamma+2\lambda_1-\nu_i$ on the left hand of (\ref{equivalencia}), we can not have the value $\gamma$ in $V_{\mathcal{N}}(H_1^*(\eta_i))$ and consequently, for such $\gamma$ we can not have $\lambda_1+k(\gamma+\lambda_1-\nu_i)\in supp\left ( \prod_{i=1}^{r}\rho_i^{\lambda_{1i}}\right )$ for any $k\geq 1$.
		
		\item[Case 3:] $\gamma+2\lambda_1-\nu_i\in\Gamma\setminus (E_{\mathcal{Z}}(H_1)+(\mathbb{N}^r)^*)$. 
		
		As $\gamma\in\langle \nu_1,\ldots ,\nu_{r+1}\rangle$ and $\gamma+\lambda_1-\nu_i\not\in\Gamma$ we may write $\gamma=\alpha_{r+1}\nu_{r+1}+\sum_{j=1\atop j\neq i}^r\alpha_j\nu_j$ with $0\leq \alpha_{r+1}\leq n_1$ and $\alpha_j\in\mathbb{N}$ for $j\in\{1,\ldots , i-1,i+1,\ldots ,r\}$. The condition $\gamma+2\lambda_1-\nu_i=(\alpha_{r+1}+2)\nu_{r+1}+\sum_{j=1\atop j\neq i}^r\alpha_j\nu_j-\nu_1\in\Gamma$ implies that $\alpha_{r+1}+2=sn_1$ for some $s\in\mathbb{N}\setminus\{0\}$, recall that $n_1\nu_{r+1}=n_1\lambda_1\in n\mathbb{N}^r$. In this way, as $n_1\geq 2$ and $\lambda_{1i}\geq n$ we get
		\begin{equation}\label{conta}
			\gamma+2\lambda_1-\nu_i=(\alpha_{r+1}+2)\nu_{r+1}+\sum_{j=1\atop j\neq i}^r\alpha_j\nu_j-\nu_1=n(\mu_1,\ldots ,\mu_r)\in n\mathbb{N}^r\ \ \mbox{with}\ \ \mu_i\geq 1.
		\end{equation}
		
		Now, for any $k\geq 1$ we express $k=2q+p$ with $q\in\mathbb{N}$, $p\in\{0,1\}$ and we obtain
		$$\begin{array}{ll}
			\lambda_1+k(\gamma +\lambda_1-\nu_i) & = \lambda_1+p\cdot (\gamma+\lambda_1-\nu_i)+q\cdot\gamma+q\cdot (\gamma+2\lambda_1-\nu_i)-q\cdot\nu_i \\
			& = \lambda_1+p\cdot (\gamma+\lambda_1-\nu_i)+q\cdot\gamma+q\cdot n\cdot (\mu_1,\ldots ,\mu_r)-q\cdot \nu_i.	
		\end{array}$$
		As $\lambda_1+p\cdot (\gamma+\lambda_1-\nu_i)\in\Gamma$ for $p\in\{0,1\}$ and, by (\ref{conta}), $\mu_i\geq 1$ we conclude (in this case) that $\lambda_1+k(\gamma+\gamma_1-\nu_i)\in\Gamma$ for every $k\geq 1$. Consequently, with the same argument for the case $\eta_i=0$ we conclude that $a_{\delta}=b_{\delta}$ for every $\delta\in E_{\mathcal{Z}}(H_1)$.
	\end{itemize}
\end{proof}

\begin{remark}\label{unicidade-coef-zariski}
	Notice that the proof of the previous theorem reveals that terms with generalized Zariski exponents cannot be eliminated by changes of coordinates belonging to the group $\widetilde{\mathcal{A}}$ and changes of coordinates belonging to $\widetilde{\mathcal{A}}_1$ keep the coefficients of such terms unchanged.
\end{remark}

By Theorem \ref{teorema} and the above remark, we can rewrite the Definition \ref{zariski-exponents} as:
\vspace{0.2cm}

\noindent{\bf Definition \ref{zariski-exponents}'.} {\it Let  $H = \left(t^n_1,\cdots, t^n_r, S(\underline{t})\right)$ be a q.o. parameterization. The set of \emph{generalized Zariski exponents} of $H$ is
	$$E_{\mathcal{Z}}(H)=\min_{\preceq}\left \{supp(S(\underline{t}))\setminus\left ( \Gamma \bigcup_{1\leq i\leq r\atop \lambda_{1i}\geq n} \left( \Gamma +2\lambda_{1}-\nu_i \right)\right )\right \}$$  
	where $\preceq$ is the product order in $\mathbb{N}^r$ and $\min_{\preceq}\{\emptyset\}=\{\underline{\infty}\}$.}
\vspace{0.2cm}

In the next section, we will consider the notion of generalized Zariski exponents to explore the $\tilde{\mathcal{A}}$-equivalence of quasi-ordinary surfaces. We will use some technical lemmas and they will be presented in Section \ref{lemas} to make the presentation more fluid.

\section{Quasi-simple surface singularities}

In this section, we will illustrate how the Zariski exponents can be used to study the $\tilde{\mathcal{A}}$-equivalence of quasi-ordinary hypersurface. In particular, we characterize quasi ordinary surfaces with one characteristic exponent that admits a countable number of distinct $\tilde{\mathcal{A}}$-classes in the same topological class. 

\begin{definition}\label{quasi-simple}
	We say that a quasi-ordinary normalized parameterization $H$ is {\emph quasi-simple} if there is a countable number of distinct $\tilde{\mathcal{A}}$-classes for q.o.h. in the same topological class of $H$, that is, we have a countable moduli.
\end{definition}

Notice that, by Example \ref{n=2} and Example \ref{normal}, any normalized q.o. parameterization with $n=2$ or $\lambda_{1}=\underline{1}=(1,\ldots ,1)$ is formally $\tilde{\mathcal{A}}$-equivalent to $(t_1^2,\ldots ,t_r^2,\underline{t}^{\lambda_{1}})$, respectively $(t_1^n,\ldots ,t_r^n,t_1\cdot\ldots\cdot t_r)$, and consequently they are quasi-simple.

In the rest of this section, we consider quasi-ordinary surfaces with one generalized characteristic exponent and $n>2$. In particular, its topological class is completely characterized by its associated semigroup which, in this case, is given by $\Gamma=\langle \nu_{1}=(n,0),\nu_{2}=(0,n),\nu_{3}=\lambda_{1}=(\lambda_{11},\lambda_{12})\rangle$.

By Proposition \ref{normalize-prop}, any quasi-ordinary surface with semigroup $\Gamma$ is formally $\tilde{\mathcal{A}}$-equivalent to a quasi-short parameterization (see Definition \ref{zariski-exponents})
\begin{equation}\label{surface}
	H=\left(t_{1}^{n}, t_{2}^{n}, \underline{t}^{\lambda_{1}}+\sum_{\delta\in E_{\mathcal{Z}}(H)}a_{\delta}\underline{t}^{\delta}u_{\delta}(\underline{t})\right)
\end{equation}
where $E_{\mathcal{Z}}(H)$ denotes the generalized Zariski exponents and $u_{\delta}(\underline{0})=1$.
In addition, Proposition \ref{normalize} allows us to normalize at most two coefficients of terms corresponding to generalized Zariski exponents, that is, if there exist $\delta_{1},\delta_{2} \in E_{\mathcal{Z}}(H)$ with $\{\delta_1-\lambda_{1},\delta_2-\lambda_{1}\}$ linearly independent, then $H$ is $\tilde{\mathcal{A}}$-equivalent to a q.o. parameterization as (\ref{surface}) with $a_{\delta_1}=a_{\delta_2}=1$.
By Remark \ref{unicidade-coef-zariski}, if there is $\delta\in E_{\mathcal{Z}}(H)\setminus\{\delta_1,\delta_2\}$ then for any $a_{\delta}\in\mathbb{C}$ we obtain a q.o. parameterization in a distinct $\tilde{\mathcal{A}}_1$-class, that is, we get an uncountable distinct orbits with respect to the $\tilde{\mathcal{A}_1}$ group. As the unique homothety that keeps the coefficient of $\underline{t}^{\lambda_1},\  \underline{t}^{\delta_1}$ and $\underline{t}^{\delta_2}$ equal to $1$ is the identity, we conclude that $H$ is not quasi-simple.

Remark that if $H$ is quasi-simple then $H$ admits at most two generalized Zariski exponents but the converse is not true as we illustrate in the following example.

\begin{example}\label{notquasisimple}
Consider the quasi-ordinary normalized parameterization $H_c=(t_1^5,t_2^5,S_c(t_1,t_2))$ where $S_c(t_1,t_2)=t_1^5t_2+t_1^5t_2^8+ct_1^5t_2^9+t_1^{10}t_2^4$. The associated semigroup to $H_c$ is given by $\Gamma=\langle (5,0),(0,5),\lambda_{1}=(5,1)\rangle$ and generalized Zariski exponents $E_{\mathcal{Z}}(H)=\{\delta_1=(5,8),\delta_2=(10,4)\}$ for any $c\in\mathbb{C}$. We will show that $H_a$ is $\tilde{\mathcal{A}}_1$-equivalent to $H_b$ if and only if $a=b$, that is, we get an uncountable distinct orbits with respect to the $\tilde{\mathcal{A}_1}$ group. As the unique homothety that keeps the coefficient of $\underline{t}^{\lambda_1},\  \underline{t}^{\delta_1}$ and $\underline{t}^{\delta_2}$ equal to $1$ is the identity, we get an uncountable distinct orbits with respect to the $\tilde{\mathcal{A}}$ group and consequently, $H_c$ is not quasi-simple.

Let us consider $(\sigma,\rho) \in \widetilde{\mathcal{A}}_1$ such that $\sigma\circ H_{a}\circ \rho^{-1}=H_{b}$.

According to presented in Section 2, since $\lambda_{11}=5$ and $\lambda_{12}=1$, we get
$$\sigma_{3}(X_1,X_2,X_3)=
X_{3}+P_{3},\ \ \ \sigma_i(X_1,X_2,X_3)=X_i+P_i,\ \ \ \rho_i(t_1,t_2)=t_i\cdot\left (1+\frac{H_a^*(P_i)}{t_i^5}\right )^{\frac{1}{5}},\ \ i=1,2
$$
with $P_1=X_1\cdot\epsilon_1+X_3\cdot\eta_1,$ $P_2=X_2\cdot \epsilon_2, P_{3}=X_{3}\cdot \epsilon_{3}+X_1X_2\cdot \eta_{3}$, such that $\epsilon_1,\epsilon_2,\epsilon_3\in\mathcal{M}_3$ and $\eta_1,\eta_3\in\mathbb{C}\{X_1,X_2,X_3\}$.

Similarly to developed in the proof of Theorem \ref{teorema}, we must have
$S_a+H_a^*(P_{3})=S_b(\rho)$
and any element in $supp(H_a^*(P_3))$ belongs to $\Gamma$, $(5,9)+(\mathbb{N}^3)^*$ or $(10,4)+(\mathbb{N}^3)^*$.

Notice that for any $(k,j)\in\mathbb{N}^2$ we have
$$\begin{array}{ll}
\rho_1^{5k}\rho_2^j & =t_1^{5k}t_2^j\cdot\left ( 1+\frac{H_a^*(P_1)}{t_1^5}\right )^k\cdot\left ( 1+\frac{H_a^*(P_2)}{t_2^5}\right )^{\frac{j}{5}}\\ 
& =t_1^{5k}t_2^j\cdot\left ( 1+H_a^*(\epsilon_1)+(t_2+t_2^8+at_2^9+t_1^5t_2^4)\cdot H_a^*(\eta_1)\right )^k\cdot\left ( 1+H_a^*(\epsilon_2)\right )^{\frac{j}{5}}.\end{array}$$ 
Since $\epsilon_1,\epsilon_2\in\mathcal{M}_3$ and denoting $H^*_a(\eta_1)(0,0)=A_0$ we get
\begin{equation}\label{exp1}\begin{array}{ll}
S_b(\rho)= & t_1^5t_2\cdot \left ( 1+H_a^*(\epsilon_1)+(t_2+t_2^8)\cdot H_a^*(\eta_1)\right )\cdot \left ( 1+H_a^*(\epsilon_2)\right )^{\frac{1}{5}}+\vspace{0.2cm}\\ & +t_1^5t_2^8\cdot (1+A_0t_2)+bt_1^5t_2^9\cdot u_1(t_1,t_2)+t_1^{10}t_2^{4}\cdot u_2(t_1,t_2)\end{array}\end{equation}
where $u_1(0,0)=u_2(0,0)=1$.

If $A_0\neq 0$ then we have $(5,2)\in supp(S_b(\rho))$ and $(5,2)\not\in supp(S_a+H^*_a(P_3))$ that is a contradiction. So, $A_0=0$ and in this way we can rewrite (\ref{exp1}) as
$$S_b(\rho)= t_1^5t_2\cdot \left ( 1+H_a^*(\epsilon_1)+(t_2+t_2^8)\cdot H_a^*(\eta_1)\right )\cdot \left ( 1+H_a^*(\epsilon_2)\right )^{\frac{1}{5}} +t_1^5t_2^8+bt_1^5t_2^9\cdot u_1(t_1,t_2)+t_1^{10}t_2^{4}\cdot u_2(t_1,t_2)$$
with $\eta_1\in\mathcal{M}_3$. Since $(0,8)\not\in supp\left ( \left ( 1+H_a^*(\epsilon_1)+(t_2+t_2^8)\cdot H_a^*(\eta_1)\right )\cdot \left ( 1+H_a^*(\epsilon_2)\right )^{\frac{1}{5}}\right )$ for any $\eta_1,\epsilon_1,\epsilon_2\in\mathcal{M}_3$, the equality $S_a+H^*_a(P_3)=S_b(\rho)$ implies that $a=b$ and consequently $H_c$ is not quasi-simple.
\end{example}

\begin{lemma}\label{3-zariskis}
	Let $H=(t_{1}^{n}, t_{2}^{n}, S(t_{1},t_{2}))$ be a quasi-short parameterization with unique characteristic exponent $\lambda_1=(\lambda_{11},\lambda_{12})$. If
	$$
	1)\ \lambda_{11}\geq\frac{4n}{n-2}\ \ \ \ \ \ \text{or}\ \ \ \ \ \ 2)\ \lambda_{12}\geq\frac{2n}{n-2} \ \ \ \ \ \ \text{or}\ \ \ \ \ \
	3)\ \lambda_{12}\geq\frac{n}{n-2}\ \mbox{and}\ \lambda_{11}\geq\frac{3n}{n-2}
	$$
	then $H$ can admit three generalized Zariski exponents.
\end{lemma}
\begin{proof}	
	Suppose that $\lambda_{11}\geq \frac{4n}{n-2}$. We get $$(n-1)\lambda_{1}+n(-4,2),\ \  (n-1)\lambda_{1}+n( -3,1),\ \  (n-1)\lambda_1+n(-2,0)\succ \lambda_{1}.$$ In fact,
	$$(n-1)\lambda_{11}-4n=(n-2)\lambda_{11}-4n+\lambda_{11}\geq\lambda_{11}\ \ \ \mbox{and}\ \ \ (n-1)\lambda_{12}+2n>\lambda_{12};$$
	$$(n-1)\lambda_{11}-3n=(n-2)\lambda_{11}-3n+\lambda_{11}\geq n+\lambda_{11}>\lambda_{11}\ \ \ \mbox{and}\ \ \ (n-1)\lambda_{12}+n> \lambda_{12};$$
	$$(n-1)\lambda_{11}-2n=(n-2)\lambda_{11}-2n+\lambda_{11}\geq 2n+\lambda_{11}>\lambda_{11}\ \ \ \mbox{and}\ \ \ (n-1)\lambda_{12}\geq \lambda_{12}.$$
	As $Z=\{(n-1)\lambda_{1}+n(-4,2), (n-1)\lambda_{1}+n(-3,1), (n-1)\lambda_{1}+n(-2,0)\}$
	is not contained in $\Gamma \bigcup_{i=1,\lambda_{1i}\geq n}^{2} \left( \Gamma +2\lambda_{1}-\nu_{i} \right)$ and $min_{\prec}(Z)=Z$, all elements of $Z$ can occur as generalized Zariski exponents.
	
	Now consider $\lambda_{12}\geq \frac{2n}{n-2}$. As $\lambda_{11}\geq \lambda_{12}$, we get $(n-1)\lambda_{1}+n(-2,0)\succ \lambda_{1}$, $(n-1)\lambda_{1}+n( 0,-2)\succ \lambda_{1}$ and $(n-1)\lambda_1+n(-1,-1)\succ \lambda_{1}$. In fact,
	$$(n-1)\lambda_{11}-2n=(n-2)\lambda_{11}-2n+\lambda_{11}\geq\lambda_{11}\ \ \ \mbox{and}\ \ \ (n-1)\lambda_{12}=(n-2)\lambda_{12}+\lambda_{12}\geq 2n+\lambda_{12}>\lambda_{12};$$
	$$(n-1)\lambda_{11}=(n-2)\lambda_{11}+\lambda_{11}\geq 2n+\lambda_{11}>\lambda_{11}\ \ \ \mbox{and}\ \ \ (n-1)\lambda_{12}-2n=(n-2)\lambda_{12}-2n+\lambda_{12}\geq \lambda_{12};$$
	$$(n-1)\lambda_{11}-n=(n-2)\lambda_{11}-n+\lambda_{11}\geq n+\lambda_{11}>\lambda_{11}\ \ \ \mbox{and}\ \ \ (n-1)\lambda_{12}-n=(n-2)\lambda_{12}-n+\lambda_{12}> \lambda_{12}.$$
	Moreover, these elements satisfy all the conditions to be generalized Zariski exponents.
	
	If $\lambda_{12}\geq\frac{n}{n-2}$ and $\lambda_{11}\geq\frac{3n}{n-2}$ then
	$$(n-1)\lambda_{11}-3n=(n-2)\lambda_{11}-3n+\lambda_{11}\geq\lambda_{11}\ \ \ \mbox{and}\ \ \ (n-1)\lambda_{12}+n>\lambda_{12};$$
	$$(n-1)\lambda_{11}-2n=(n-2)\lambda_{11}-2n+\lambda_{11}>\lambda_{11}\ \ \ \mbox{and}\ \ \ (n-1)\lambda_{12}>\lambda_{12};$$
	$$(n-1)\lambda_{11}-n=(n-2)\lambda_{11}-n+\lambda_{11}>\lambda_{11}\ \ \ \mbox{and}\ \ \ (n-1)\lambda_{12}-n=(n-2)\lambda_{12}-n+\lambda_{12}\geq\lambda_{12}.$$
	So, we can verify that $(n-1)\lambda_1+n(-3,1), (n-1)\lambda_{1}+n(-2,0)$ and $(n-1)\lambda_1+n(-1,-1)$ can be possible generalized Zariski exponents.
\end{proof}

By the previous lemma, to a quasi-short parameterization to admit at most two generalized Zariski exponents, which is a necessary condition to be quasi-simple, we must consider $\lambda_{11}<\frac{3n}{n-2}$, $\lambda_{12}<\frac{2n}{n-2}$ or $\lambda_{12}<\frac{n}{n-2}$ if $\frac{3n}{n-2}\leq\lambda_{11}< \frac{4n}{n-2}$.

In the following proposition we will analize the case $\lambda_{12}<\frac{n}{n-2}$ and $\frac{3n}{n-2}\leq\lambda_{11}< \frac{4n}{n-2}$ simultaneously with the possibility $\lambda_{12}=0$ for any situation.

\begin{proposition}\label{proposicao1}
	Let $H$ be a quasi-short parameterization with $\Gamma=\langle (n,0), (0,n), \lambda_{1}\rangle$. If
	$$\lambda_{12}=0\ \ \ \ \ \mbox{or}\ \ \ \ \ \lambda_{12}<\frac{n}{n-2}\ \mbox{and}\ \frac{3n}{n-2}\leq\lambda_{11}< \frac{4n}{n-2}$$
	then $H$ is quasi-simple if and only if $n\in\{3,4\}$ and $(\lambda_{11},\lambda_{12})\neq \left ( \frac{3n}{n-2},0\right )$. In this case, $H$ is formally $\tilde{\mathcal{A}}$-equivalent to
	$$
	\left( t_{1}^{n},t_{2}^{n},\underline{t}^{\lambda_{1}}+a\underline{t}^{(n-1)\lambda_{1}+n(-2,i)} +b\underline{t}^{(n-1)\lambda_{1}+n(-3,j)}\right)
	$$
	where $a, b \in \{0,1\}$, $i,j\in \mathbb{N}$ with $a=0$ if $i\geq j$.
\end{proposition}
\begin{proof}
	As $H$ is normalized, if $\lambda_{12}=0$ we must have $n<\lambda_{11}$. But $\lambda_{11}<\frac{4n}{n-2}$ so $n<6$.

	On the other hand, for $\lambda_{12}\neq 0$, if $n\geq 6$ and $\lambda_{11}<\frac{4n}{n-2}$ we get $\lambda_{11}< 6\leq n$. As $\frac{3n}{n-2}\leq\lambda_{11}$
	we have	
	\begin{equation}\label{3z0}
		(n-1)\lambda_{1}+n(-3,1)\succ \lambda_{1},\ \ (n-1)\lambda_1+n(-2,0)\succ\lambda_{1}\ \ \mbox{and}\ \ (n-2)\lambda_{1}+n(-2,1)\succ\lambda_{1}
	\end{equation}
	because 
	$$(n-1)\lambda_{11}-3n=(n-2)\lambda_{11}-3n+\lambda_{11}\geq \lambda_{11}\ \ \mbox{and}\ \ (n-1)\lambda_{12}+n>\lambda_{12},$$
	$$(n-1)\lambda_{11}-2n=(n-2)\lambda_{11}-2n+\lambda_{11}> \lambda_{11}\ \ \mbox{and}\ \ (n-1)\lambda_{12}\geq\lambda_{12},$$
	$$(n-2)\lambda_{11}-2n\geq n\geq\lambda_{11}\ \ \mbox{and}\ \ (n-2)\lambda_{12}+n>\lambda_{12}.$$
	Moreover, since $\lambda_{11}<n$ and $\lambda_{12}<n$, we can verify that all three elements in (\ref{3z0}) can be generalized Zariski exponents for $H$, so $H$ can not be quasi-simple.
	
	If $n=5$, then $5=\frac{3n}{n-2}\leq\lambda_{11}<\frac{4n}{n-2}=\frac{20}{3}$ that is, $\lambda_{11}\in\{5,6\}$. For $\lambda_{11}=6$, independently $\lambda_{12}=0$ or not, we get 
	$$(n-1)\lambda_{1}+n(-3,1)\succ \lambda_{1},\ \ (n-1)\lambda_1+n(-2,0)\succ\lambda_{1}\ \ \mbox{and}\ \ (n-2)\lambda_{1}+n(-2,2)\succ\lambda_{1}.$$
	 Similar to (\ref{3z0}), we conclude that these values are possible generalized Zariski exponents for $H$ and the parameterization is not quasi-simple.
    For $\lambda_{11}=5=n$ we must have $\lambda_{12}=1$, that is, $\lambda_{1}=(5,1)$ and Example \ref{notquasisimple} indicates that $H$ can not be quasi-simple.
	
	In this way, to obtain a quasi-simple parameterization we must have $3\leq n\leq 4$. 
	
	Notice that $\lambda_{12}<\frac{n}{n-2}\leq n$ and $n<\frac{3n}{n-2}\leq\lambda_{11}$. In this way the exponents in the quasi-short parameterization $H$ does not belong to $\Gamma\cup\{\Gamma+2\lambda_{1}-(n,0)\}$.
	
	By Remark \ref{std-rep}, any $\delta\not\in\Gamma$ can be expressed as
	$\delta =c_3\lambda_{1}+n(c_1,c_2)$ with $0\leq c_3<n$ where $c_1<0$ or $c_2<0$. In addition, in order to $\delta\succ\lambda_{1}$ we must have $c_3\geq 2$.
	
	If $n=4$ and $c_3=2$, the condition $\lambda_1\prec\delta=2\lambda_1+4(c_1,c_2)$ implies $c_2\geq 0$ and $c_1\geq -1$. But such a condition gives us $\delta\in \Gamma+2\lambda_{1}-(4,0)$ that is not a possible exponent in a quasi-short parameterization.
	
	So, for $3\leq n\leq 4$ the exponents to be considered in $H$ are
	$(n-1)\lambda_{1}+n(c_1,c_2)\succ\lambda_{1}$, i.e.,
	$$c_1\geq -\frac{n-2}{n}\cdot \lambda_{11}>-\frac{n-2}{n}\cdot \frac{4n}{n-2}=-4\ \ \mbox{and}\ \ c_2\geq -\frac{n-2}{n}\cdot \lambda_{12}>-\frac{n-2}{n}\cdot\frac{n}{n-2}=-1.$$ 
	
	If $c_1=-1$ and $c_2\geq 0$ we obtain an element in $\Gamma+2\lambda_{1}-(n,0)$, that is not a possible exponent in a quasi-short parameterization. 
	
	If $\lambda_{12}=0$ and $\lambda_{11}=\frac{3n}{n-2}$ then $n_1\neq n$ and $\Gamma\neq\langle (n,0),(0,n),\lambda_{1}\rangle$. So, we can exclude this possibility.
	
	In this way, for $3\leq n\leq 4$, we can consider a quasi-short parameterization given by
	$$
	\left( t_{1}^{n},t_{2}^{n},\underline{t}^{\lambda_{1}}+\sum_{k\geq 0}a_{k}t^{(n-1)\lambda_{1}+n(-2,k)}+\sum_{l\geq 0}b_{l}t^{(n-1)\lambda_{1}+n(-3,l)}  \right ),
	$$
	with $l>0$ if $\lambda_{12}=0$ and $\lambda_{11}\neq \frac{3n}{n-2}$.
	In particular $E_{\mathcal{Z}}(H)\subset\{(n-1)\lambda_{1}+n(-2,k),\ (n-1)\lambda_{1}+n(-3,l);\ k, l\in\mathbb{N}\}\cup\{ \underline{\infty}\}$.
	
	Notice that for $l\leq k$ we get $(n-1)\lambda_{1}+n(-2,k)=(n-1)\lambda_{1}+n(-3,l)+n(1,k-l)\in (n-1)\lambda_{1}+n(-3,l)+\Gamma\setminus\{(0,0)\}$, that is, $(n-1)\lambda_{1}+n(-2,k)$ is not a generalized Zariski exponent.	
	
	Suppose that $E_{\mathcal{Z}}(H)=\{\delta=(n-1)\lambda_1+n(-2,i)\}$ for some $i\geq 0$, that is, $a_{l}=0$ for every $l\geq 0$. As any $\gamma\succ\delta$ with $\gamma\in Q_1=\mathbb{Z}\lambda_{1}+n\mathbb{Z}^2$ is such that $\gamma\in \Gamma\cup(\Gamma+2\lambda_{1}-\nu_1)\cup\{\delta+n(0,k); k\geq 1\}$ we can apply (possibly infinite many times) Proposition \ref{normalize-prop} and Lemma \ref{elim-zariski1} in order to eliminate all terms $\underline{t}^{\gamma}$ with $\gamma\succ\delta=(n-1)\lambda_{1}+n(-2,i)$ and we obtain $(t_1^n,t_2^n,\underline{t}^{\lambda_{1}}+a_{i}\underline{t}^{(n-1)\lambda_{1}+n(-2,i)})$. In addition, by Proposition \ref{normalize}, we conclude that, in this case, $H$ is formally $\tilde{\mathcal{A}}$-equivalent to 
		$\left ( t_1^n,t_2^n,\underline{t}^{\lambda_{1}}+\underline{t}^{(n-1)\lambda_{1}+n(-2,i)}\right ) $.
	
	If $E_{\mathcal{Z}}(H)=\{\delta=(n-1)\lambda_1+n(-3,i)\}$ for some $i\geq 0$, that is, $a_{k}=0$ for every $k<l$ then for any $\gamma\in Q_1=\mathbb{Z}\lambda_{1}+n\mathbb{Z}^2$ with $\gamma\succ\delta$ we have $\gamma\in \Gamma\cup(\Gamma+2\lambda_{1}-\nu_1)\cup\{\delta+\Gamma\setminus\{(0,0)\}\}$. Applying (possibly infinite many times) Proposition \ref{normalize-prop} and Lemma \ref{elim-zariski1} we can to eliminate all terms $\underline{t}^{\gamma}$ with $\gamma\succ\delta=(n-1)\lambda_{1}+n(-2,i)$ and, by Proposition \ref{normalize}, $H$ is formally $\tilde{\mathcal{A}}$-equivalent to 
		$\left ( t_1^n,t_2^n,\underline{t}^{\lambda_{1}}+\underline{t}^{(n-1)\lambda_{1}+n(-2,i)}\right ) $.
	
Finally, suppose that $E_{\mathcal{Z}}(H)=\{\delta_1=(n-1)\lambda_1+n(-2,i),\delta_2=(n-1)\lambda_1+n(-3,j)\}$ for some $0\leq i<j$ then 
		$$\{\gamma\succ\delta_1;\ \gamma\in Q_1\}\cup \{\gamma\succ\delta_2;\ \gamma\in Q_1\}\subset  \Gamma\cup(\Gamma+2\lambda_{1}-\nu_1)\cup\{\delta_1+n(0,k),\ k>0\}\cup\{\delta_{2}+n(0,k),\ k>0\}.$$ In this situation, as $\{\delta_1-\lambda_{1}, \delta_{2}-\lambda_{1}\}$ is a linearly independent set, we can apply (possibly infinite many times) Proposition \ref{normalize-prop}, Lemma \ref{elim-zariski2}, Proposition \ref{normalize} and we conclude that $H$ is formally $\tilde{\mathcal{A}}$-equivalent to 
		$\left ( t_1^n,t_2^n,\underline{t}^{\lambda_{1}}+\underline{t}^{(n-1)\lambda_{1}+n(-2,i)}+\underline{t}^{(n-1)\lambda_{1}+n(-3,j)}\right ) $.
	
\end{proof}

By Lemma \ref{3-zariskis} and Proposition \ref{proposicao1}, the next step is to consider the cases:
\begin{equation}\label{condicoes}
	0<\lambda_{12}\leq \lambda_{11}< \frac{2n}{n-2}\ \ \ \  \text{or}\ \ \ \  0<\lambda_{12}< \frac{2n}{n-2}\leq \lambda_{11} < \frac{3n}{n-2}.
\end{equation}

\begin{proposition}\label{proposicao2}
	Let $H=(t_{1}^{n}, t_{2}^{n}, S(t_{1},t_{2}))$ be a quasi-short parameterization with $0<\lambda_{12}\leq \lambda_{11}< \frac{2n}{n-2}$, then $H$ is quasi-simple if and only if 
	\begin{enumerate}
		\item[a)] $\Gamma_H=\langle (n,0), (0,n), (1,1)\rangle$ and in this case  
		$H$ is formally $\tilde{\mathcal{A}}$-equivalent to $ \left( t_{1}^{n},t_{2}^{n},t_1t_2\right)$, that is, a normal surface;
		
		\item[b)] $\Gamma_H=\langle (n,0), (0,n), (2,1)\rangle$ for $6\leq n\leq 7$ or $\Gamma_H=\langle (5,0), (0,5), (3,1)\rangle$. In this case 
		$H$ is formally $\tilde{\mathcal{A}}$-equivalent to $\left( t_{1}^{n},t_{2}^{n},\underline{t}^{\lambda_{1}}+a\underline{t}^{(n-2)\lambda_{1}+n(-1,i)}+b\underline{t}^{(n-1)\lambda_{1}+n(-1,j)}\right)$ 
		where $a,b\in\{0,1\}$, $i, j\in\mathbb{N}$ with $b=0$ if $i\leq j$.
		
		\item[c)] $\Gamma_H=\langle (5,0), (0,5), (2,1)\rangle$  or $\Gamma_H=\langle (4,0), (0,4), (\lambda_{11},1)\rangle$ where $2\leq \lambda_{11}\leq 3$. In this case 
		$H$ is formally $\tilde{\mathcal{A}}$-equivalent to $\left( t_{1}^{n},t_{2}^{n},\underline{t}^{\lambda_{1}}+a\underline{t}^{(n-1)\lambda_{1}+n(-1,i)}\right)$ 
		where $a\in\{0,1\}$ and $i\in\mathbb{N}$.
		
		\item[d)] $\Gamma_H=\langle (5,0), (0,5), (2,2)\rangle$  or $\Gamma_H=\langle (4,0), (0,4), \lambda_{1})\rangle$ where $\lambda_{1}\in\{(3,2),(3,3)\}$. In this case 
		$H$ is formally $\tilde{\mathcal{A}}$-equivalent to $$\left( t_{1}^{n},t_{2}^{n},\underline{t}^{\lambda_{1}}+a\underline{t}^{(n-1)\lambda_1+n(-1,-1)}+b\underline{t}^{(n-1)\lambda_{1}+n(i,-1)}+c\underline{t}^{(n-1)\lambda_{1}+n(-1,j)}\right)$$ 
		where $a,b,c\in\{0,1\}$, $i, j\in\mathbb{N}$ with $b=c=0$ if $a=1$.
		
		\item[e)] $\Gamma_H=\langle (3,0), (0,3), (\lambda_{11},\lambda_{12})\rangle$ where $1\leq \lambda_{12}\leq\lambda_{11}\in\{2,3,4,5\}$ with $\lambda_{1}\neq (3,3)$. In this case, 
		$H$ is $\tilde{\mathcal{A}}$-equivalent to $(t_1^3,t_2^3,\underline{t}^{\lambda_{1}})$ if $\lambda_{12}< 3$ and $(t_1^3,t_2^3,\underline{t}^{\lambda_{1}}+a\underline{t}^{2\lambda_{1}+3(-1,-1)});\ a\in\{0,1\}$ if $3\leq\lambda_{12}.$
	\end{enumerate}  
	
\end{proposition}
\begin{proof}
	As $0<\lambda_{12}\leq\lambda_{11}<\frac{2n}{n-2}$, for $n\geq 4$ we get $\lambda_{12}\leq\lambda_{11}<n$ and for $n\geq 6$ we obtain that $\lambda_{11}\leq 2$.
	
	If $\lambda_{11}=1$ then we must have $\lambda_{12}=1$ and the proposition follows by Example \ref{normal}.
	
	So, in what follows we suppose that $\lambda_{11}\geq 2$ and we split the cases according to the possible value for $n>2$.
	
	{\bf Case $n\geq 6$:} Notice that in this case we must have $1\leq \lambda_{12}\leq\lambda_{11}=2<n$. In particular, if a parameterization $H=(t_1^n,t_2^n,S(\underline{t}))$ is quasi-short then  $\lambda_{1}\prec\gamma\in supp(S(\underline{t})-\underline{t}^{\lambda_{1}})$ does not belong to $\Gamma$.
	
	For $\lambda_{12}=2$ we can verify that
	\begin{equation}\label{zz}
		(n-1)\lambda_{1}+n(-1,0),\ \ (n-1)\lambda_{1}+n(0,-1)\ \ \mbox{and}\ \ (n-2)\lambda_{1}+n(-1,1)\end{equation} are possible generalized Zariski exponents.
	In addition, for $n\geq 8$ and $1\leq\lambda_{12}\leq 2$, we have that
	\begin{equation}\label{zzz}
	(n-1)\lambda_{1}+n(-1,0),\ \ (n-2)\lambda_{1}+n(-1,1)\ \ \mbox{and}\ \ (n-3)\lambda_{1}+n(-1,2)
	\end{equation}
	can be considered as generalized Zariski exponents and, as we have remarked, we do not have quasi-simple parameterization. Recall that if $\lambda_{1}=(2,2)$ then $n$ must be odd in order to obtain the value semigroup $\langle (n,0),(0,n),\lambda_{1}\rangle$.
	
	So, in this case, it is sufficient to consider $6\leq n\leq 7$, $\lambda_{11}=2$ and $\lambda_{12}=1$.
	
	By Remark \ref{std-rep}, any $\gamma\not\in\Gamma$ can be expressed as $\gamma=c_3\lambda_1+n(c_1,c_2)$ with $0\leq c_3<n$ where $c_1<0$ or $c_2<0$. Notice that for $6\leq n\leq 7$ if $c_3\leq n-3$, to obtain $\gamma=c_3(2,1)+n(c_1,c_2)\succ\lambda_{1}=(2,1)$ we must have $c_1, c_2\in\mathbb{N}$, that is, $\gamma\in\Gamma$. Moreover,
	$$\begin{array}{c}
		(n-1)\lambda_{1}+n(c_1,c_2)\succ\lambda_{1}\\
		(n-2)\lambda_{1}+n(c_1,c_2)\succ\lambda_{1}
	\end{array}\ \ \Leftrightarrow\ \ c_1\geq -1\ \mbox{and}\ c_2\geq 0.$$
	Consequently, in this case, any quasi-short parameterization can be given by
	\begin{equation}\label{forma1}
		H=\left( t_{1}^{n},t_{2}^{n},\underline{t}^{\lambda_{1}}+\sum_{k\geq 0}a_{k}\underline{t}^{(n-2)\lambda_{1}+n(-1,k)}+\sum_{l\geq 0}b_{l}\underline{t}^{(n-1)\lambda_{1}+n(-1,l)} \right)
	\end{equation}
	and $E_{\mathcal{Z}}(H)\subset\{(n-2)\lambda_{1}+n(-1,k),\ (n-1)\lambda_{1}+n(-1,l);\ k, l\in\mathbb{N}\}\cup\{ \underline{\infty}\}$. Notice that if $k\leq l$ then $(n-1)\lambda_1+n(-1,l)=(n-2)\lambda_1+n(-2,k)+\lambda_1+n(0,l-k)$, that is, we have at most one generalized Zariski exponent. 
	
	Similarly to Proposition \ref{proposicao1} we conclude that $H$ is formally $\tilde{\mathcal{A}}$-equivalent to 
		\begin{equation}\label{caso2}
			\left( t_{1}^{n},t_{2}^{n},\underline{t}^{\lambda_{1}}+a\underline{t}^{(n-2)\lambda_{1}+n(-1,i)}+b\underline{t}^{(n-1)\lambda_{1}+n(-1,j)}\right)\end{equation}
		where $a,b\in\{0,1\}$, $i, j\in\mathbb{N}$ with $b=0$ if $i\leq j$.
	
	{\bf Case $n=5$:} In this case we can have $1\leq \lambda_{12}\leq\lambda_{11}\leq\frac{2n}{n-2}=\frac{10}{3}<n=5$. 
	
	Recall that, by Remark \ref{std-rep}, any $\gamma\not\in\Gamma$ can be expressed as $\gamma=c_3\lambda_1+5(c_1,c_2)$ with $0\leq c_3<5$ where $c_1<0$ or $c_2<0$. In order to $\gamma\succ\lambda_{1}$ we must have $c_3\in\{3,4\}$ with $c_1= -1$ or $c_2=-1$. 
	
	If $2\leq\lambda_{12}\leq\lambda_{11}=3$ then the elements in (\ref{zz}) are possible generalized Zariski exponents and we do not obtain a quasi-simple surface.
	
	If $\lambda_{1}=(3,1)$ then we can verify that the possible generalized Zariski exponents are $4\lambda_{1}+5(-1,c_2)$ and $3\lambda_{1}+5(-1,c_1)$ with $c_1,c_2\in\mathbb{N}$. So, we can consider quasi-short parameterization given as (\ref{forma1}) and consequently, $H$ is formally $\tilde{\mathcal{A}}$-equivalent to (\ref{caso2}).
	
	For $\lambda_{11}=2$ we have that
	$$\gamma=3(2,\lambda_{12})+5(c_1,c_2)\succ (2,\lambda_{12})=\lambda_{1}\ \ \Leftrightarrow\ \ c_1,c_2\in\mathbb{N},$$
	that is, $\gamma\in\Gamma$. In this way, possible generalized Zariski exponent is $4(2,\lambda_{12})+5(c_1,c_2)\not\in\Gamma$ with $c_1\geq -1, c_2\geq -1$ if $\lambda_{12}=2$ and $c_1=-1$, $c_2\geq 0$ for $\lambda_{12}=1$. Consequently, a quasi-short parameterization can be given by $H=(t_1^n,t_2^n,S(\underline{t}))$ with
	\begin{equation}\label{forma2}
		S(\underline{t})=\underline{t}^{\lambda_{1}}+\sum_{k\geq 0}a_{k}\underline{t}^{(n-1)\lambda_{1}+n(-1,k)}\ \ \mbox{if}\ \ \lambda_{12}=1
	\end{equation}
	or
	\begin{equation}\label{forma3}
		S(\underline{t})=\underline{t}^{\lambda_{1}}+a\underline{t}^{(n-1)\lambda_{1}+n(-1,-1)}+\sum_{k\geq 0}a_{k}\underline{t}^{(n-1)\lambda_{1}+n(k,-1)}+\sum_{l\geq 0}b_{l}\underline{t}^{(n-1)\lambda_{1}+n(-1,l)} \ \  \mbox{if}\ \lambda_{12}=2.
	\end{equation}
	
	If $H$ is as (\ref{forma2}) and $E_{\mathcal{Z}}(H)=\{(n-1)\lambda_{1}+n(-1,i)\}$ then, similarly to Proposition \ref{proposicao1}, we conclude that, $H$ is formally $\tilde{\mathcal{A}}$-equivalent to 
		$\left ( t_1^5,t_2^5,\underline{t}^{\lambda_{1}}+\underline{t}^{4\lambda_{1}+5(-1,i)}\right )$ with $i\in\mathbb{N}$.
		
If in (\ref{forma3}) we have
		\begin{enumerate}
			\item $a\neq 0$, then $E_{\mathcal{Z}}(H)=\{\delta=(n-1)\lambda_{1}+n(-1,-1)\}$ and any $\gamma\in Q_1=\mathbb{Z}\lambda_{1}+n\mathbb{Z}^2$ with $\gamma\succ\delta$ belongs to the set $\Gamma\cup\{\delta+k\nu_1;\ k>0\}\cup\{\delta+k\nu_2;\ k>0\}$.
			\item $a=0$, $a_{c_2}=0$ for all $c_2\geq 0$ and $E_{\mathcal{Z}}(H)=\{\delta=(n-1)\lambda_{1}+n(i,-1)\}$ for some $i\geq 0$, then $\gamma\in Q_1=\mathbb{Z}\lambda_{1}+n\mathbb{Z}^2$ with $\gamma\succ\delta$ belongs to the set $\Gamma\cup\{\delta+k\nu_1;\ k>0\}$.
			\item $a=0$, $a_{c_1}=0$ for all $c_1\geq 0$ and $E_{\mathcal{Z}}(H)=\{\delta=(n-1)\lambda_{1}+n(-1,j)\}$ for some $j\geq 0$, then $\gamma\in Q_1=\mathbb{Z}\lambda_{1}+n\mathbb{Z}^2$ with $\gamma\succ\delta$ belongs to the set $\Gamma\cup\{\delta+k\nu_2;\ k>0\}$.	
		\end{enumerate} 
	
In the above cases we have one generalized Zariski exponent $\delta$ and, as Proposition \ref{proposicao1}, we obtain that $H$ is formally $\tilde{\mathcal{A}}$-equivalent to $(t_1^n, t_2^n, \underline{t}^{\lambda_{1}}+\underline{t}^{\delta})$.
	
If $a=0$ and $E_{\mathcal{Z}}(H)=\{\delta_{1}=(n-1)\lambda_{1}+n(i,-1),\delta_{2}=(n-1)\lambda_{1}+n(-1,j)\}$ for some $i,j\in\mathbb{N}$, then 	$\{\gamma\succ\delta_1;\ \gamma\in Q_1\}\cup \{\gamma\succ\delta_2;\ \gamma\in Q_1\}\subset  \Gamma\cup\{\delta_1+n(k,0),\ k>0\}\cup\{\delta_{2}+n(0,k),\ k>0\}$. In this case, we can apply (possibly infinite many times) Proposition \ref{normalize-prop}, Lemma \ref{elim-zariski3}, Proposition \ref{normalize} and we conclude that $H$ is formally $\tilde{\mathcal{A}}$-equivalent to 
		$\left ( t_1^n,t_2^n,\underline{t}^{\lambda_{1}}+\underline{t}^{(n-1)\lambda_{1}+n(i,-1)}+\underline{t}^{(n-1)\lambda_{1}+n(-1,j)}\right ) $. 
	
	{\bf Case $n=4$:} Recall that we get $0<\lambda_{12}\leq\lambda_{1}<\frac{2n}{n-2}=4$, that is $\lambda_{12}\leq\lambda_{11}<4=n$ and any possible generalized Zariski exponent is expressed as $\delta=3\lambda_{1}+4(c_1,c_2)\succ\lambda_{1}$ with $c_1<0$ or $c_2<0$.
	
	If $\lambda_{12}=1$ then we must have $c_1= -1$ and $c_2\geq 0$, that is, we can consider the quasi-short parameterization given as (\ref{forma2}). In addition, if $E_{\mathcal{Z}}(H)\neq\{\underline{\infty}\}$ we can proceed as Proposition \ref{proposicao1} and to conclude that $H$ is formally $\tilde{\mathcal{A}}$-equivalent to $\left ( t_1^4,t_2^4,\underline{t}^{\lambda_{1}}+\underline{t}^{3\lambda_{1}+4(-1,i)}\right )$ with $i\in\mathbb{N}$.
	
	Recall we do not have $\lambda_{11}=\lambda_{12}=2$ otherwise, the value semigroup is not $\langle (n,0), (0,n), \lambda_{11}\rangle$. Excluding this case, if $1< \lambda_{12}\leq\lambda_{11}\leq 3$ then $\delta\in E_{\mathcal{Z}}(H)$ if and only if 
	$\delta = 3\lambda_{1}+4(c_1,c_2)$ with $c_1=-1$ or $c_2=-1$. So, the quasi-short parameterization is given as (\ref{forma3}) and we can proceed in the same way as the case $n=5$ and $\lambda_{1}=(2,2)$.
	
	{\bf Case $n=3$:} Notice that we get $0<\lambda_{12}\leq\lambda_{1}<\frac{2n}{n-2}=6$, that is $\lambda_{12}\leq\lambda_{11}\in\{2,3,4,5\}$ with the restriction $\lambda_{12}<3$ if $\lambda_{11}=3$ in order to obtain $\Gamma=\langle (3,0), (0,3), \lambda_{1}\rangle$.
	
	Any $\lambda_{1}\prec\gamma\in Q_1$ with $\gamma\notin\Gamma$ is expressed, by Remark \ref{std-rep}, as $\gamma=2(\lambda_{11},\lambda_{12})+3(c_1,c_2)\succ\lambda_{1}=(\lambda_{11},\lambda_{12})$ and considering the possible values for $\lambda_{1i}$ we obtain 
	$$c_{i}\geq 0\ \ \mbox{if}\ \ \lambda_{1i}<3\ \ \ \ \ \mbox{and}\ \ \ \ \ c_i\geq -1\ \ \mbox{if}\ \ \lambda_{1i}\geq 3.$$
	As the possible generalized Zariski exponents do not belong to $\Gamma\bigcup_{1\leq i\leq 2\atop \lambda_{1i}\geq 3}(\Gamma+2\lambda_{1}-\nu_i)$, we get
	$$E_{\mathcal{Z}}(H)=\{\underline{\infty}\}\ \ \mbox{if}\ 0<\lambda_{12}< 3\ \ \ \ \ \mbox{and}\ \ \ \ \ \ E_{\mathcal{Z}}(H)\subset\{2\lambda_{1}+3(-1,-1),\underline{\infty}\}\ \ \mbox{if}\ 3\leq\lambda_{12}.$$ If $E_{\mathcal{Z}}(H)=\{2\lambda_{1}+3(-1,-1)\}$ then $\gamma\succ 2\lambda_{1}+3(-1,-1)$ belongs to $\Gamma\bigcup_{1\leq i\leq 2\atop \lambda_{1i}\geq 3}(\Gamma+2\lambda_{1}-\nu_i)$ and it can be eliminable. In addition, Proposition \ref{normalize} allows us to normalize the coefficient of the term with the generalized Zariski exponent. Consequently, $H$ is formally $\tilde{\mathcal{A}}$-equivalent to
	$$(t_1^3,t_2^3,\underline{t}^{\lambda_{1}})\ \ \mbox{for}\ \ \lambda_{12}< 3\ \ \ \ \ \ \mbox{or}\ \ \ \ \ \ (t_1^3,t_2^3,\underline{t}^{\lambda_{1}}+a\underline{t}^{2\lambda_{1}+3(-1,-1)});\ a\in\{0,1\}\ \ \mbox{for}\ \ 3\leq\lambda_{12}.$$
\end{proof}

Now we will consider $0<\lambda_{12}< \frac{2n}{n-2}\leq \lambda_{11} < \frac{3n}{n-2}$. 

Notice that for $n\geq 5$ we obtain $\lambda_{11}<\frac{3n}{n-2}\leq 5\leq n$ and $\lambda_{12}<\frac{2n}{n-2}<n$. So, in this case, in a quasi-short parameterization $H=(t_1^n,t_2^n,S(\underline{t}))$ with $n\geq 5$ we get $\gamma\in supp(S(\underline{t})-\underline{t}^{\lambda_{1}})$ implies $\gamma\not\in\Gamma$ and, by Remark \ref{std-rep}, any $\gamma\not\in\Gamma$ can be expressed as $\gamma=c_3\lambda_1+n(c_1,c_2)$ with $0\leq c_3<n$ where $c_1<0$ or $c_2<0$.

For $6\leq n\leq 7$ we obtain $3\leq\lambda_{11}\leq 4$ and $\lambda_{11}=3$ for $n\geq 8$. In this way, for $n\geq 6$ all the elements in (\ref{zzz})
can be considered as generalized Zariski exponents. 

If $n=5$ then $3<\frac{2n}{n-2}\leq\lambda_{11}<\frac{3n}{n-2}=5$, that is, $\lambda_{11}=4$ and $0<\lambda_{12}<\frac{2n}{n-2}=\frac{10}{3}$. In this case, $4\lambda_1+5(-2,1), 4\lambda_{1}+5(-1,0)$ and $3\lambda_{1}+5(-1,1)$ can be generalized Zariski exponents.

So, for $0<\lambda_{12}<\frac{2n}{n-2}\leq \lambda_{11}<\frac{3n}{n-2}$ and $n\geq 5$ it is possible to get three generalized Zariski exponents and consequently, we can not have quasi-simple surfaces.

\begin{proposition}\label{proposicao3}
	Let $H=(t_{1}^{n}, t_{2}^{n}, S(t_{1},t_{2}))$ be a quasi-short parameterization with $0<\lambda_{12}< \frac{2n}{n-2}\leq \lambda_{11}<\frac{3n}{n-2}$, then $H$ is quasi-simple if and only if 
	\begin{enumerate}
		\item[a)] $\Gamma_H=\langle (3,0), (0,3), \lambda_1\rangle$  where 
		$1\leq \lambda_{12}\leq 5< \lambda_{11}\leq 8$ with $\lambda_{1}=(\lambda_{11},\lambda_{12})\neq (6,3)$. In this case $H$ is formally $\tilde{\mathcal{A}}$-equivalent to
		$$\left (t_1^3,t_2^3,\underline{t}^{\lambda_{1}}+a\underline{t}^{2\lambda_{1}+3(-2,i)}\right ); \ a\in\{0,1\}\ \mbox{and}\ i\geq 0\ \ \mbox{if}\ \ \lambda_{12}\in\{1,2\}\ \ \mbox{or}$$
		$$\left (t_1^3,t_2^3,\underline{t}^{\lambda_{1}}+a\underline{t}^{2\lambda_{1}+3(-1,-1)}+b\underline{t}^{2\lambda_{1}+3(-2,i)}\right ); \ a,b\in\{0,1\}, i\geq -1\\ \ \mbox{if}\ \ 3\leq\lambda_{12}\leq 5,$$
		with $a=0$ if $i=-1$ and $b=1$.	
		
		\item[b)] $\Gamma_H=\langle (4,0), (0,4), \lambda_{1}\rangle$ where $\lambda_{12}\in\{1,2,3\}$ and $\lambda_{11}\in\{4,5\}$ with $(\lambda_{11},\lambda_{12})\neq(4,2)$. In this case, 
		$H$ is formally $\tilde{\mathcal{A}}$-equivalent to 
		$$\left (t_1^4,t_2^4,\underline{t}^{\lambda_{1}}+a\underline{t}^{3\lambda_{1}+4(-2,i)}\right );\ a\in\{0,1\},\ i\in\mathbb{N}\ \ \mbox{if}\ \ \lambda_{12}=1\ \ \mbox{or} $$
		$$
	\left (t_1^4,t_2^4,\underline{t}^{\lambda_{1}}+a\underline{t}^{3\lambda_{1}+4(-2,i)}+b\underline{t}^{3\lambda_{1}+4(j,-1)}\right );\ \ \mbox{if}\ \ \lambda_{12}\in\{2,3\}$$
	with $a,b\in\{0,1\}$, $i, j\geq -1$ where $b=0$ if $i=-1$ and $a=1$.
	\end{enumerate} 	
\end{proposition}
\begin{proof}
	By the comment before the proposition, it is sufficient to consider $3\leq n\leq 4$.
	
	{\bf Case $n=3$:} The conditions $0<\lambda_{12}< \frac{2n}{n-2}\leq \lambda_{11}<\frac{3n}{n-2}$ give us $0<\lambda_{12}\leq 5$ and $6\leq \lambda_{11}\leq 8$. Recall that to obtain $\Gamma=\langle (3,0),(0,3),\lambda_{1}\rangle$ we must have $\lambda_{1}\neq (6,3)$.
	
	Taking $\gamma\not\in\Gamma$ we can write $\gamma=c_3\lambda_1+4(c_1,c_2)$ with $0\leq c_3<3$ where $c_1<0$ or $c_2<0$. To obtain $\gamma\succ\lambda_{1}$ it is sufficient to consider $c_3=2$, $c_1\geq -2$ and 
	$$c_2\geq 0\ \ \mbox{for}\ \ \lambda_{12}\in\{1,2\}\ \ \ \ \  \mbox{or}\ \ \ \ \  c_2\geq -1\ \ \mbox{for}\ \ \lambda_{12}\in\{3,4,5\}.$$
	
	In this case, an exponent $\gamma\succ\lambda_{1}$ in a quasi-short parameterization is such that
	$$\gamma\not\in\Gamma\cup(\Gamma+2\lambda_{1}-(3,0))\ \ \mbox{if}\ \ \lambda_{12}<3\ \ \ \ \mbox{or} \ \ \ \ \ \gamma\not\in\Gamma\cup(\Gamma+2\lambda_{1}-(3,0))\cup (\Gamma+2\lambda_{1}-(0,3))\ \ \mbox{if}\ \ \lambda_{12}\geq 3.$$
	So, any quasi-simple parameterization with $n=3$ and $0<\lambda_{12}< \frac{2n}{n-2}\leq \lambda_{11}<\frac{3n}{n-2}$ is formally $\tilde{\mathcal{A}}$-equivalent to
	$$\left (t_1^3,t_2^3,\underline{t}^{\lambda_{1}}+\sum_{k\geq 0}a_{k}\underline{t}^{2\lambda_{1}+3(-2,k)}\right )\ \ \mbox{if}\ \ \lambda_{12}\in\{1,2\}\ \ \mbox{or}$$
	$$\left (t_1^3,t_2^3,\underline{t}^{\lambda_{1}}+a\underline{t}^{2\lambda_{1}+3(-1,-1)}+\sum_{k\geq -1}a_{k}\underline{t}^{2\lambda_{1}+3(-2,k)}\right )\ \ \mbox{if}\ \ \lambda_{12}\in\{3,4,5\}.$$
	
	As we proceeded in Proposition \ref{proposicao1}, we conclude that $H$ is formally $\tilde{\mathcal{A}}$-equivalent to 
	$$\left (t_1^3,t_2^3,\underline{t}^{\lambda_{1}}+a\underline{t}^{2\lambda_{1}+3(-2,i)}\right )\ a\in\{0,1\},\ i\in\mathbb{N}\ \ \ \mbox{if}\ \ 0<\lambda_{12}<3\ \ \mbox{or}$$
	$$\left (t_1^3,t_2^3,\underline{t}^{\lambda_{1}}+a\underline{t}^{2\lambda_{1}+3(-1,-1)}+b\underline{t}^{2\lambda_{1}+3(-2,i)}\right ),\ a,b\in\{0,1\},\ i\geq -1\ \ \ \mbox{if}\ \ 3\leq\lambda_{12}\leq 5,$$
	with $a=0$ if $i=-1$.
	
	{\bf Case $n=4$:} As $0<\lambda_{12}< \frac{2n}{n-2}\leq \lambda_{11}<\frac{3n}{n-2}$ we get $1\leq\lambda_{12}\leq 3$ and $4\leq\lambda_{11}\leq 5$ with $\lambda_{1}\neq (4,2)$. So, in a quasi-short parameterization it is sufficient consider exponents $\gamma\succ\lambda_{1}$ such that $\gamma\not\in\Gamma\cup (\Gamma+2\lambda_{1}-\nu_1)$.
	
	By Remark \ref{std-rep}, any $\gamma\not\in\Gamma$ can be expressed $\gamma=c_3\lambda_1+4(c_1,c_2)$ with $0\leq c_3<4$ where $c_1<0$ or $c_2<0$. In order to $\gamma\succ\lambda_{1}$ we have to consider $c_3\in\{2,3\}$.
	
	For $c_3=2$ we have
	$\gamma=2(\lambda_{11},\lambda_{12})+4(c_1,c_2)\succ (\lambda_{11},\lambda_{12})$ if and only if $c_1\geq -1$ and $c_2\geq 0$. In this case, $\gamma\in \Gamma\cup (\Gamma+2\lambda_{1}-\nu_1)$ and we can discard such exponents in a quasi-short parameterization.
	
	Taking $c_3=3$ we have
	$\gamma=3(\lambda_{11},\lambda_{12})+4(c_1,c_2)\succ (\lambda_{11},\lambda_{12})$ if and only if $c_1\geq -2$ and 
	$$c_2\geq 0\ \ \mbox{for}\ \ \lambda_{12}=1\ \ \ \ \  \mbox{or}\ \ \ \ \  c_2\geq -1\ \ \mbox{for}\ \ \lambda_{12}\in\{2,3\}.$$
	If $c_1\geq -1$ and $c_2\geq 0$ we obtain $\gamma\in \Gamma\cup (\Gamma+2\lambda_{1}-\nu_1)$. So, any quasi-short parameterization with $n=4$ and $0<\lambda_{12}<\frac{2n}{n-2}=4\leq\lambda_{11}<\frac{3n}{n-2}=6$ is $\tilde{\mathcal{A}}$-equivalent to
	\begin{equation}\label{4-1}
		\left (t_1^4,t_2^4,\underline{t}^{\lambda_{1}}+\sum_{k\geq 0}a_{k}\underline{t}^{3\lambda_{1}+4(-2,k)}\right )\ \ \mbox{if}\ \ \lambda_{12}=1;\end{equation}
	\begin{equation}\label{4-2}
	\left (t_1^4,t_2^4,\underline{t}^{\lambda_{1}}+\sum_{k\geq -1}b_{k}\underline{t}^{3\lambda_{1}+4(-2,k)}+\sum_{l\geq -1}c_{l}\underline{t}^{3\lambda_{1}+4(l,-1)}\right )\ \ \mbox{if}\ \ \lambda_{12}\in\{2,3\}.\end{equation}
	In addition,  
	\begin{enumerate}
		\item If $E_{\mathcal{Z}}(H)=\{\delta=3\lambda_{1}+4(-2,i)\}$ with $i\in\mathbb{N}$ in (\ref{4-1}) or
		\item If $E_{\mathcal{Z}}(H)=\{\delta=3\lambda_{1}+4(-2,i)\}$ with $i\geq -1$ in (\ref{4-2}) or
			\item If $E_{\mathcal{Z}}(H)=\{\delta=3\lambda_{1}+4(i,-1)\}$ with $i\geq -1$ in (\ref{4-2})
	\end{enumerate}
	then, by Lemma \ref{elim-zariski1} and Proposition \ref{normalize-prop}, we conclude that $H$ is formally $\tilde{\mathcal{A}}$-equivalent to 
		$\left (t_1^4,t_2^4,\underline{t}^{\lambda_{1}}+\underline{t}^{\delta}\right )$. 
	
	On the other hand, if it is not the case and $E_{\mathcal{Z}}(H)\neq\{\underline{\infty}\}$ then in (\ref{4-2}) we have $b_{-1}=0$ and $E_{\mathcal{Z}}=\{3\lambda_{1}+4(-2,i), 3\lambda_{1}+4(j,-1)\}$. In this case, we proceed as Proposition \ref{proposicao1} and $H$ is formally $\tilde{\mathcal{A}}$-equivalent to 
	$\left (t_1^4,t_2^4,\underline{t}^{\lambda_{1}}+\underline{t}^{3\lambda_{1}+4(-2,i)}+\underline{t}^{3\lambda_{1}+4(j,-1)}\right )$.
	
\end{proof}

We summarize the results of this section in the following theorem.

\begin{theorem} Let $H=(t_1^n,t_2^n, S(t_1,t_2))$ be a normalized q.o. parameterization with value semigroup $\Gamma=\langle (n,0),(0,n),\lambda_{1}=(\lambda_{11},\lambda_{12})\rangle$. Then $H$ is quasi-simple if and only if  we have one of the following cases:
	\begin{enumerate}
		\item[a)] $n=2$. In this case, $H$ is formally $\tilde{\mathcal{A}}$-equivalent to $\left( t_{1}^{2},t_{2}^{2},\underline{t}^{\lambda_{1}}\right)$.
		
		\item[b)] $\lambda_{1}=(1,1)$. In this case, $H$ is formally $\tilde{\mathcal{A}}$-equivalent to $\left( t_{1}^{n},t_{2}^{n},t_1t_2\right)$, i.e., a normal surface.
		
		\item[c)] $n=3$ and
		\begin{enumerate}
			\item[c.1)]  $1\leq\lambda_{12}\leq\lambda_{11}\in\{2,3,4,5\}$ with $\lambda_{1}\neq (3,3)$. In this case 
			$H$ is formally $\tilde{\mathcal{A}}$-equivalent to $$(t_1^3,t_2^3,\underline{t}^{\lambda_{1}})\ \ \mbox{if}\ \  \lambda_{12}\in\{1,2\}\ \ \ \ \ \mbox{or}\ \ \ \ \ (t_1^3,t_2^3,\underline{t}^{\lambda_{1}}+a\underline{t}^{2\lambda_{1}+3(-1,-1)});\ a\in\{0,1\}\ \ \mbox{if}\ \ \ 3\leq\lambda_{12}.$$
			\item[c.2)]	$1\leq\lambda_{12}\leq 5< \lambda_{11}\leq 8$ with $\lambda_{1}\neq (3,6)$. In this case $H$ is formally $\tilde{\mathcal{A}}$-equivalent to
			$$\left (t_1^3,t_2^3,\underline{t}^{\lambda_{1}}+a\underline{t}^{2\lambda_{1}+3(-2,i)}\right )\ a\in\{0,1\}\ \mbox{and}\ i\geq 0\ \ \mbox{if}\ \ \lambda_{12}\in\{1,2\}\ \ \mbox{or}$$
			$$\left (t_1^3,t_2^3,\underline{t}^{\lambda_{1}}+a\underline{t}^{2\lambda_{1}+3(-1,-1)}+b\underline{t}^{2\lambda_{1}+3(-2,i)}\right )\ a,b\in\{0,1\}, i\geq -1\\ \ \mbox{if}\ \ \lambda_{12}\in\{3,4,5\},$$
			with $a=0$ if $i=-1$ e $b=1$.	
			
			\item[c.3)] $0\leq \lambda_{12}\leq 2$ and $9\leq\lambda_{11}\leq 11$ with $(\lambda_{11},\lambda_{12})\neq (9,0)$. In this case $H$ is formally $\tilde{\mathcal{A}}$-equivalent to		
			$$\left( t_{1}^{3},t_{2}^{3},\underline{t}^{\lambda_{1}}+a\underline{t}^{2\lambda_{1}+3(-2,i)} +b\underline{t}^{2\lambda_{1}+3(-3,j)}\right)
			$$ where $a, b \in \{0,1\}$, $i, j\in \mathbb{N}$ with $a=0$ if $i\geq j$.	
			
		\end{enumerate}
		
		\item[d)] $n=4$ and
		\begin{enumerate}
			\item[d.1)] $\lambda_{1}=(\lambda_{11},1)$ where $2\leq \lambda_{11}\leq 3$. In this case 
			$H$ is formally $\tilde{\mathcal{A}}$-equivalent to $$\left( t_{1}^{4},t_{2}^{4},\underline{t}^{\lambda_{1}}+a\underline{t}^{3\lambda_{1}+4(-1,i)}\right)\ \  
			\mbox{where}\ a\in\{0,1\}\ \mbox{and}\ i\in\mathbb{N}.$$
			\item[d.2)] $2\leq\lambda_{12}\leq\lambda_{11}\leq 3$ with $(\lambda_{11},\lambda_{12})\neq (2,2)$. In this case 
			$H$ is formally $\tilde{\mathcal{A}}$-equivalent to $$\left( t_{1}^{4},t_{2}^{4},\underline{t}^{\lambda_{1}}+a\underline{t}^{3\lambda_{1}+4(-1,-1)}+b\underline{t}^{3\lambda_{1}+4(i,-1)}+c\underline{t}^{3\lambda_{1}+4(-1,j)}\right)$$ 
			where $a,b,c\in\{0,1\}$, $i, j\in\mathbb{N}$ and $b=c=0$ if $a=1$.
			\item[d.3)] $\lambda_{12}\in\{1,2,3\}$ and $\lambda_{11}\in\{4,5\}$ with $(\lambda_{11},\lambda_{12})\neq(4,2)$. In this case, 
			$H$ is formally $\tilde{\mathcal{A}}$-equivalent to 
			$$\left (t_1^4,t_2^4,\underline{t}^{\lambda_{1}}+a\underline{t}^{3\lambda_{1}+4(-2,i)}\right );\ a\in\{0,1\},\ i\in\mathbb{N}\ \ \mbox{if}\ \ \lambda_{12}=1\ \ \mbox{or} $$
			$$
			\left (t_1^4,t_2^4,\underline{t}^{\lambda_{1}}+a\underline{t}^{3\lambda_{1}+4(-2,i)}+b\underline{t}^{3\lambda_{1}+4(j,-1)}\right );\ \ \mbox{if}\ \ \lambda_{12}\in\{2,3\}$$
			with $a,b,c\in\{0,1\}$, $i,j\geq -1$ whrere $b=0$ if $i=-1$ and $a=1$.
			\item[d.4)] $\lambda_{12}\in\{0,1\}$ and $\lambda_{11}\in\{6,7\}$ with $(\lambda_{11},\lambda_{12})\neq (6,0)$. In this case, $H$ is formally $\tilde{\mathcal{A}}$-equivalent to
			$$
			\left( t_{1}^{4},t_{2}^{4},\underline{t}^{\lambda_{1}}+a\underline{t}^{3\lambda_{1}+4(-2,i)} +b\underline{t}^{3\lambda_{1}+4(-3,j)}\right)
			$$ 
			where $a, b \in \{0,1\}$, $i, j\in \mathbb{N}$ with $a=0$ if $i\geq j$.			
		\end{enumerate}
		
		\item[e)] $n=5$ and $\lambda_{1}\in\{(2,1),(3,1),(2,2)\}$. In this case 
		$H$ is formally $\tilde{\mathcal{A}}$-equivalent to	
		$$\left( t_{1}^{5},t_{2}^{5},\underline{t}^{(2,1)}+a\underline{t}^{4(2,1)+5(-1,i)}\right);\ \mbox{where}\ a\in\{0,1\}\ \mbox{and}\ i\in\mathbb{N};$$		
		$$\left( t_{1}^{5},t_{2}^{5},\underline{t}^{(3,1)}+a\underline{t}^{4(3,1)+5(-1,i)}+b\underline{t}^{3(3,1)+5(-1,j)}\right); \ a,b\in\{0,1\},\ i,j\in\mathbb{N}\ \mbox{with}\ b=0\ \mbox{if}\ j\leq i;$$	
		$$\left( t_{1}^{5},t_{2}^{5},\underline{t}^{(2,2)}+a\underline{t}^{4(2,2)+5(-1,-1)}+b\underline{t}^{4(2,2)+5(i,-1)}+c\underline{t}^{4(2,2)+5(-1,j) } \right);\ \mbox{where} \ a,b,c\in\{0,1\},\ $$ $i, j\in\mathbb{N}$ with $b=c=0$ if $a=1$.	
		
		\item[f)] $6\leq n\leq 7$ and $\lambda_{1}=(2,1)$. In this case 
		$H$ is formally $\tilde{\mathcal{A}}$-equivalent to $$\left( t_{1}^{n},t_{2}^{n},\underline{t}^{(2,1)}+a\underline{t}^{(n-1)(2,1)+n(-1,i)}+b\underline{t}^{(n-2)(2,1)+n(-1,j)}\right)$$ 
		where $a,b\in\{0,1\}$, $i, j\in\mathbb{N}$ with $b=0$ if $j\leq i$.
	\end{enumerate}
\end{theorem}
\begin{proof}
	It follows from Example \ref{n=2}, Example \ref{normal}, Proposition \ref{proposicao1}, Proposition \ref{proposicao2} and Proposition \ref{proposicao3}.
\end{proof}

\section{Technical Lemmas}\label{lemas}

In this section, we present some technical lemmas that make use of notations and concepts presented in Section 2 concerning dominant exponent and differentials $2$-forms. We use these lemmas in Proposition \ref{proposicao1}, Proposition \ref{proposicao2}, and Proposition \ref{proposicao3} to obtain normal forms of quasi-simple surfaces concerning $\tilde{\mathcal{A}}$ group.

\begin{lemma}\label{elim-zariski1}
	If $H=\left ( t_1^n,t_2^n, \underline{t}^{\lambda_{1}}+a\underline{t}^{\delta}u(\underline{t})\right )$ is a q.o. parameterization where $a\neq 0$ and $u(\underline{0})=1$ then $\underline{t}^{\delta+\gamma}$ is eliminable for any $\gamma\in\Gamma\setminus\{(0,0)\}$.
\end{lemma}
\begin{proof}
	By Proposition \ref{elimination} it is sufficient to exhibit $\omega=\sum_{i=1}^{r+1}(-1)^{r+1-i}P_idX_{1}\wedge\cdots\wedge\widehat{dX_{i}}\wedge \cdots\wedge dX_{r+1} \in \Omega^{r}$ where $P_i$ is as described in (\ref{cond-mud}) for all $i=1,\ldots,r$ such that $\mathcal{V}(\omega)=\delta+\gamma+(n,n)$.
	
	Let us take 
	\begin{equation}\label{omega0}
		\omega_0=\frac{1}{n}\left (s_1X_1dX_2\wedge dX_3 +s_2X_2dX_1\wedge dX_3+\frac{(s_1\lambda_{11}-s_2\lambda_{12})}{n}X_3dX_1\wedge dX_2\right ).
	\end{equation}
	Considering the map $\Psi_H$ given in (\ref{Psi}) we get
	$$\Psi_H(\omega_0)=\left ( s_2(\delta_2-\lambda_{12})-s_1(\delta_{1}-\lambda_{11})\right )a\underline{t}^{\delta+(n,n)}u(\underline{t})+a\underline{t}^{\delta+(n,n)}(s_2t_2u_2(\underline{t})-s_1t_1u_1({\underline{t})}),$$
	where $u_i(\underline{t})$ denotes the derivative of $u(\underline{t})$ with respect to $t_i$.
	
	Since $t_iu_i(\underline{t})$ is not a unit, for any $\alpha\in supp(a\underline{t}^{\delta+(n,n)}(s_2t_2u_2(\underline{t})-s_1t_1u_{1}(\underline{t})))$ we have $\alpha\succ\delta+(n,n)$. Recall that $\delta_1\neq\lambda_{11}$ or $\delta_{2}\neq\lambda_{12}$ so, we can choose $s_1, s_2\in\mathbb{C}$ in a such way that $\mathcal{V}(\omega_0)=\delta+(n,n)$.
	
	Now, for any $\gamma\in\Gamma\setminus\{(0,0)\}$ we take $\epsilon\in\mathcal{M}_{3}$ such that $\mathcal{V}(\epsilon)=\gamma$ and, in this way, $\omega=\epsilon\omega_0\in\Omega^r$ satisfies the conditions (\ref{cond-mud}) and $\mathcal{V}(\omega)=\delta+\gamma+(n,n)$. Consequently, by Proposition \ref{elimination}, $\underline{t}^{\delta+\gamma}$ is eliminable by $\tilde{\mathcal{A}}$-action.
\end{proof}

Notice that if $\delta\in\Gamma\bigcup_{1\leq i\leq 2\atop \lambda_{1i}\geq n}(\Gamma+2\lambda_{1}-\nu_i)$ then the above lemma is a particular case of Proposition \ref{normalize-prop}. On the other hand if $E_{\mathcal{Z}}(H)=\{\delta\}$ then Lemma \ref{elim-zariski1} allows us to eliminate terms with exponent in $\delta+\Gamma\setminus\{(0,0)\}$ and it can be considered as the counterpart to the quasi-ordinary case of an elimination criterion of terms in plane curve parameterizations proved by Zariski (see Section 2.3, Chapter III in \cite{zariski-book}).

In the same way, the relevance of the next result corresponds to the case $E_{\mathcal{Z}}(H)=\{\delta_{1},\delta_{2}\}$.

\begin{lemma}\label{elim-zariski2}
	Let $H=\left ( t_1^n,t_2^n, S(\underline{t})=\underline{t}^{\lambda_{1}}+\sum_{j\geq 0}a_j\underline{t}^{\delta_1+j\cdot\nu_1}+\sum_{j\geq 0}b_j\underline{t}^{\delta_2+j\cdot\nu_1}\right )$ be a q. o. parameterization where $a_0\neq 0\neq b_0$, $\nu_1=(n,0)$ with $\min\{S(\underline{t})-\underline{t}^{\lambda_{1}}\}=\{\delta_1,\delta_2\}$. If $\{\delta_1-\lambda_{1},\delta_{2}-\lambda_{1}\}\subset \mathbb{R}^2$ is a linearly independent set then $\underline{t}^{\delta_1+j\cdot \nu_1}$ and $\underline{t}^{\delta_2+j\cdot\nu_1}$ are eliminable for any $j>0$. The same is true if we change $\nu_1$ by $\nu_2=(0,n)$.
\end{lemma}
\begin{proof}
	As the previous lemma, it is sufficient to guarantee the existence of differential $2$-forms that admit dominant exponent $\delta_1+(n,n)+j\cdot \nu_1$ and $\delta_2+(n,n)+j\cdot \nu_1$ for any $j>0$.
	
	Considering $\omega_0\in\Omega^r$ given as (\ref{omega0}) and the map $\Psi_H$ given in (\ref{Psi}) we get
	$$\begin{array}{ll}
		\Psi_H(\omega_0) = & \left ( s_2\cdot (\delta_{12}-\lambda_{12})-s_1\cdot (\delta_{11}-\lambda_{11})\right )a_0\underline{t}^{\delta_{1}+(n,n)}+ \\
		& 
		+\sum_{j>0}\left ( s_2\cdot (\delta_{12}-\lambda_{12})-s_1\cdot (\delta_{11}+jn-\lambda_{11})\right )a_j\underline{t}^{\delta_{1}+j\cdot\nu_1+(n,n)}+ \\
		& \left ( s_2\cdot (\delta_{22}-\lambda_{12})-s_1\cdot (\delta_{21}-\lambda_{11})\right )b_0\underline{t}^{\delta_{2}+(n,n)}+ \\
		& 
		+\sum_{j>0}\left ( s_2\cdot (\delta_{22}-\lambda_{12})-s_1\cdot (\delta_{21}+jn-\lambda_{11})\right )b_j\underline{t}^{\delta_{2}+j\cdot\nu_1+(n,n)}.
	\end{array}$$ 
	
	As $\{\delta_1-\lambda_{1},\delta_{2}-\lambda_{1}\}$ is a linearly independent set, the linear system equations
	$$(*)\ \left \{ \begin{array}{l} (\delta_{12}-\lambda_{12})\cdot Z - (\delta_{11}-\lambda_{11})\cdot W = \frac{1}{a_0} \\
		(\delta_{22}-\lambda_{12})\cdot Z - (\delta_{21}-\lambda_{11})\cdot W = 0 
	\end{array}\right .\ \ \mbox{and}\ \ (**)\ \left \{ \begin{array}{l} (\delta_{12}-\lambda_{12})\cdot Z - (\delta_{11}-\lambda_{11})\cdot W = 0 \\
		(\delta_{22}-\lambda_{12})\cdot Z - (\delta_{21}-\lambda_{11})\cdot W = \frac{1}{b_0} 
	\end{array}\right .$$
	admit solutions. Taking a solution $(Z,W)=(s_2,s_1)$ for the system $(*)$ and substituting in $\omega_0$ we get the differential $2$-form $\omega_1$ with
	$$\Psi_H(\omega_1)=\underline{t}^{\delta_{1}+(n,n)}
	+\sum_{j>0}a'_j\underline{t}^{\delta_{1}+j\cdot\nu_1+(n,n)}+\sum_{j>0}b'_j\underline{t}^{\delta_{2}+j\cdot\nu_1+(n,n)}$$
	and considering a solution $(Z,W)=(s_2,s_1)$ for the system $(**)$ and substituting in $\omega_0$ we obtain the differential $2$-form $\omega_2$ with
	$$\Psi_H(\omega_2)=\underline{t}^{\delta_{2}+(n,n)}
	+\sum_{j>0}a''_j\underline{t}^{\delta_{1}+j\cdot\nu_1+(n,n)}+\sum_{j>0}b''_j\underline{t}^{\delta_{2}+j\cdot\nu_1+(n,n)}$$
	for some $a'_j, a''_j, b'_j, b''_j\in\mathbb{C}$.
	
	In this way, there are $c_{k}, d_{k}\in\mathbb{C}$ for $k>0$ such that
	\begin{equation}\label{diff-x1}
			\mathcal{V}\left (\Psi_H\left (\omega_1+\sum_{k>0}c_kX_1^k\omega_2\right )\right )=\delta_{1}+(n,n)\ \ \mbox{and}\ \ \mathcal{V}\left (\Psi_H\left (\omega_2+\sum_{k>0}d_kX_1^k\omega_1\right )\right )=\delta_{2}+(n,n).\end{equation}
	
	As, for any $j>0$, the differential $2$-forms $$X_1^j\left (\omega_1+\sum_{k>0}c_kX_1^k\omega_2\right )\ \  \mbox{and}\ \ X_1^j\left (\omega_2+\sum_{k>0}d_kX_1^k\omega_1\right )$$
	satisfy the conditions (\ref{cond-mud}). By Proposition \ref{elimination}, any term $\underline{t}^{\delta_1+j\cdot \nu_1}$ and $\underline{t}^{\delta_2+j\cdot \nu_1}$ are eliminable by $\tilde{\mathcal{A}}$-action.
	
	If  $H=\left ( t_1^n,t_2^n, \underline{t}^{\lambda_{1}}+\sum_{j\geq 0}a_j\underline{t}^{\delta_1+j\cdot\nu_2}+\sum_{j\geq 0}b_j\underline{t}^{\delta_2+j\cdot\nu_2}\right )$ then interchange $X_1$ by $X_2$ in (\ref{diff-x1}) we obtain that any term $\underline{t}^{\delta_1+j\cdot \nu_2}$ and $\underline{t}^{\delta_2+j\cdot \nu_2}$ with $j>0$ are eliminable by $\tilde{\mathcal{A}}$-action.
	
\end{proof}

We note that the next lemma is a particular case of the Proposition \ref{normalize-prop} if $\lambda_{1i}\geq n$, but it gives us relevant information for the other cases.

\begin{lemma}\label{elim-zariski3}
	Let $H=\left ( t_1^n,t_2^n, \underline{t}^{\lambda_{1}}+\sum_{j\geq 0}a_j\underline{t}^{\delta_1+j\cdot\nu_1}+\sum_{j\geq 0}b_j\underline{t}^{\delta_2+j\cdot\nu_2}\right )$ be a q. o. parameterization where $\delta_{1}=(n-1)\lambda_{1}+n(i,-1), \delta_{2}=(n-1)\lambda_{1}+n(-1,j)$ with $a_0\neq 0\neq b_0$ and $i, j\in\mathbb{N}$. We have that the terms $\underline{t}^{\delta_1+j\cdot \nu_1}$ and $\underline{t}^{\delta_2+j\cdot\nu_2}$ are eliminable for any $j>0$. 
\end{lemma}
\begin{proof}
	We will present differential $2$-forms with dominant exponents $\delta_1+(n,n)+j\cdot \nu_1$ and $\delta_2+(n,n)+j\cdot \nu_2$ for any $j>0$ and, in this way, the result follows by Proposition \ref{elimination}.
	
	We take $\omega_0\in\Omega^r$ given as (\ref{omega0}) and we consider the expansion $\Psi_H(\omega_0)$. As $\{\delta_1-\lambda_{1},\delta_{2}-\lambda_{1}\}$ is a linearly independent set, the linear systems $(*)$ and $(**)$ admit solutions that give us differentials $2$-forms $\omega_1$ and $\omega_2$ such that
	$$\Psi_H(\omega_1)=\underline{t}^{\delta_{1}+(n,n)}
	+\sum_{j>0}a'_j\underline{t}^{\delta_{1}+j\cdot\nu_1+(n,n)}+\sum_{j>0}b'_j\underline{t}^{\delta_{2}+j\cdot\nu_2+(n,n)}$$
	$$\Psi_H(\omega_2)=\underline{t}^{\delta_{2}+(n,n)}
	+\sum_{j>0}a''_j\underline{t}^{\delta_{1}+j\cdot\nu_1+(n,n)}+\sum_{j>0}b''_j\underline{t}^{\delta_{2}+j\cdot\nu_2+(n,n)}$$
	for some $a'_j, a''_j, b'_j, b''_j\in\mathbb{C}$.
	
	So, we can determine $c_{k}, d_{k}\in\mathbb{C}$ for $k>0$ such that
	$$\mathcal{V}\left (\Psi_H\left (\omega_1+\sum_{k>0}c_kX_2^k\omega_2\right )\right )=\delta_{1}+(n,n)\ \ \mbox{and}\ \ \mathcal{V}\left (\Psi_H\left (\omega_2+\sum_{k>0}d_kX_1^k\omega_1\right )\right )=\delta_{2}+(n,n).$$
	In addition, for any $j>0$, the differential $2$-forms $$X_1^j\left (\omega_1+\sum_{k>0}c_kX_2^k\omega_2\right )\ \  \mbox{and}\ \ X_2^j\left (\omega_2+\sum_{k>0}d_kX_1^k\omega_1\right )$$
	satisfy the conditions (\ref{cond-mud}) consequently, by Proposition \ref{elimination}, any term $\underline{t}^{\delta_1+j\cdot \nu_1}$ and $\underline{t}^{\delta_2+j\cdot \nu_2}$ are eliminable by $\tilde{\mathcal{A}}$-action.
\end{proof}

\noindent{\bf Acknowledgment.}
The authors are grateful to the anonymous referee for the 
suggestions that have improved this work, mainly by motivating us to include the Example \ref{notquasisimple}.
\vspace{0.2cm}

\noindent{\bf Data Availability Statement.} The authors declare that the data supporting the findings of this study are available within its references
\vspace{0.2cm}

\noindent{\bf Conflict of Interest Statement.} The authors certify that they have NO involvement in any organization or entity with any financial interest in the subject matter in this manuscript.

\end{document}